\documentclass[12pt]{amsart}
\usepackage{amsfonts,latexsym}
\usepackage{amsmath,amsthm,amssymb,amsfonts,amscd,color}
\usepackage{tikz}
\usepackage{pictexwd,dcpic}

\usepackage{amsmath,tikz-cd}
\usepackage{url}
\usepackage{mathtools}
\usepackage{hyperref}
\usepackage{breqn}
\usepackage{geometry}
\theoremstyle{plain} \numberwithin{equation}{section}
\newtheorem{theo}{Theorem}[section]
\newtheorem{corollary}[theo]{Corollary}
\newtheorem{proposition}[theo]{Proposition}
\newtheorem{lemma}[theo]{Lemma}

\theoremstyle{definition}
\newtheorem{definition}{Definition}
\newtheorem*{example}{Example}
\newtheorem*{rema}{Remark}

\def\Z{\mathbb Z}
\def\CC{\mathbb C}
\def\RR{\mathbb R}

\def\N{\mathbb N}
\def\mcF{\mathcal F}
\def\mcrF{\mathring{\mathcal F}}
\def\mrF{\mathring {F}}

\newcommand{\red}{}

\begin{document}

\title{Equivariant cohomology ring of open torus manifolds with locally standard actions}
\author{Yueshan Xiong}
\address{School of Mathematics and Statistics, Huazhong University of Science and Technology, Wuhan, 430074, P.R. China}
\email{yueshan\_xiong@yahoo.com}

\author{Haozhi Zeng}
\address{School of Mathematics and Statistics, Huazhong University of Science and Technology, Wuhan, 430074, P.R. China}
\email{zenghaozhi@icloud.com}

\begin{abstract}
The notation of torus manifolds were introduced by A. Hattori and M. Masuda. Toric manifolds, quasitoric manifolds, topological toric manifolds, toric origami manifolds \red{and $b$-symplectic toric manifolds} are typical examples of torus manifolds with locally standard action. Recently, L. Yu introduced a nice notion topological face ring $\mathbf{k}[Q]$, a generalization of Stanley-Reisener ring, for a nice manifold with corners $Q$. L. Yu applied polyhedral product technique developed by A. Bahri, M. Bendersky, F. Cohon and S. Gilter to show that the equivariant cohomology ring $H^*_T(M)$ of an open torus manifold $M$ with locally standard action is isomorphic to the topological face ring of $M/T$
under the assumption that the free part of the action is a trivial torus bundle. In this paper we show that Yu's formula holds for any open torus manifolds with locally standard action by a different appoach. In addition using our method we give an explicit formula for equivariant Stiefel-Whitney classes and Pontrjagin classes of open torus manifolds with locally standard action. \end{abstract}
\maketitle

\section{Introduction}
The notion of torus manifolds were introduced by A. Hattori and M. Masuda in \cite{ha-ma-03}. A torus manifold is a smooth connected closed manifold $M$ of dimension $2n$ with an effective smooth action of an $n$-dimensional torus group $T$. We say the $T$-action on $M$ is locally standard if this action locally looks like the standard $T^n$-action on $\CC^n.$ Here the standard $T^n$-action on $\CC^n$ means 
\[(t_1,\dots,t_n)\cdot(z_1,\dots,z_n)=(t_1z_1,\dots, t_nz_n),\]
where $t_i, z_i\in\CC$ and $|t_i|=1$ for $1\le i\le n.$ There are plentiful examples of
torus manifolds with locally standard action   in toric topology such as toric manifolds \cite{co-li-sc-11}, \cite{fu-93}, quasitoric manifolds introduced by M. W. Davis and T. Januszkiewicz  \cite{da-ja-91}, co\"orientable toric origami manifolds introduced by A. Cannas da Silva, V. Guillemin and A. R. Pires \cite{ca-gu-pi-11}, \cite{ho-pi-13},  topological toric manifolds introduced by H. Ishida, Y. Fukukawa and M. Masuda \cite{is-fu-ma-13}, quasitoric manifolds with holes introduced by M. Poddar and S. Sarkar \cite{po-sa-15} and $b$-symplectic toric manifolds introduced by V. Guillemin, E. Miranda, A. R. Pires and G. Scott \cite{gu-mi-pi-sc-15} and so on. In addition, M. Masuda applied torus manifold generalized the classic Pick's formula  \cite{ma-99} and his work pointed out that it is promising to applied torus manifold to study combinatorics.

If $M$ is torus manifold with locally standard action, then its orbit space $Q=M/T$ is a nice manifold with corners. Roughly speaking, a nice manifold with corners locally looks like an open subset of the positive cone $\RR_{\ge 0}^n.$ By T. Yoshida's work we know that we can reconstruct $M$ from $Q$ and extra data up to equivariant homeomorphism. Here the extra data mean an element  $c\in H^2(Q;\mathbb{Z}^n)$, which determine a principal $T^n$-bundle $E$ over $Q$, and a map $\lambda: \{F_1,\dots, F_m\}\to \text{Hom}(S^1,T)$, where $\{F_1,\dots, F_m\}$ is the set of facets of $Q$. More detail can be found in \cite{yo-11}. Hence a natural question is how to read the algebraic topological information from the triple $(Q,c,\lambda)$. The answer is well-known for toric manifolds (see \cite{co-li-sc-11} or \cite{fu-93}), quasitoric manifolds \cite{da-ja-91} and more general case that each face of $Q$ is acyclic \cite{ma-pa-06}, but for general torus manifolds with locally standard action the answer is unknown.

Recently, L. Yu introduced a nice notion topological face ring $\mathbf{k}[Q]$, which is a generalization of Stanley-Reisener ring,  for a nice manifold with corners $Q$ \cite{yu-21}.
Yu showed that if $M$ is an open torus manifold with locally standard action under the assumption that the free part of the action on $M$ is a trivial torus bundle, then $H_T^*(M)\cong\mathbf{k}[M/T]$. Yu's proof based on and generalized polyhedral product technique, developed by  A. Bahri, M. Bendersky, F. Cohon and S. Gilter's in \cite{ba-be-co-gi-10}, so the assumption is necessary. However the following example shows that there are many torus manifolds with locally standard action whose free part of the torus action is not a trivial torus bundle. \red{For example}, the connected sum of the toric manifold $\mathbb{CP}^2$ with a principal $T^2$ bundle $P$ over $T^2$ along a principal orbit is a torus manifold $M$ with locally standard action but the free part of the $T^2$ action is not a non-trivial $T^2$ principal bundle. More precisely, 
\[M = \mathbb{CP}^2 \#_{T^2} P=\Big(\mathbb{CP}^2\setminus{(T^2(p) \times \mathring{\mathbb{D}}^2)}\Big)\bigcup_{T^2\times S^1} \Big(P\setminus({T_1^2 \times \mathring{\mathbb{D}}^2})\Big),\]
where $T^2(p)\cong T^2$ is the orbit space of the point $p \in \mathbb{CP}^2$ with trivial isotropy group, $T_1^2\cong T^2$ is a fiber of the principal $T^2$ bundle $P$ over $T^2$ and \red{$\mathring{\mathbb{D}}^2\cong\{(x,y)\in\mathbb{R}^2: x^2+y^2<1\}$}. The orbit space $M/T^2$ is a nice manifold with corners (see figure \ref{fig:connected sum}).
 
\begin{figure}[h]
	\begin{center}
		\begin{tikzpicture}[scale=0.6]
			\pgfsetfillopacity{1}
			\filldraw [fill=yellow!50, draw=blue,line width=0.3mm] (-5.5,2.25) -- (-4,-2) -- (-7,-0.2) -- (-5.5,2.25);
			
			\draw (3.5,0) .. controls (3.5,2) and (1.5,2.5) .. (0,2.5);
			\draw[rotate=180] (-3.5,0) .. controls (-3.5,2) and (-1.5,2.5) .. (0,2.5);
			\draw[draw=white,double=black] (0,2.5) .. controls (-1.5,2.5) and (-2.8,0.6) .. (-5.5,0.6);
			\draw [draw=white,double=black,yscale=-1] (0,2.5) .. controls (-1.5,2.5) and (-2.8,0.6) .. (-5.5,0.6);
			
			\draw (-2,.2) .. controls (-1.5,-0.3) and (-1,-0.5) .. (0,-.5) .. controls (1,-0.5) and (1.5,-0.3) .. (2,0.2);
			
			\draw (-1.75,0) .. controls (-1.5,0.3) and (-1,0.5) .. (0,.5) .. controls (1,0.5) and (1.5,0.3) .. (1.75,0);	
			
			\draw(0,.5) arc (-90:90:-0.3 and 1);
			\draw[dashed] (0,.5) arc (-90:90:0.3 and 1);
			
			\draw(0,-2.5) arc (-90:90:-0.3 and 1);
			\draw[dashed] (0,-2.5) arc (-90:90:0.3 and 1);
			
			\filldraw[fill=white!50, draw = yellow, dashed] (-5.5,-0.6) arc (-90:90:-0.3 and 0.6) (-5.5,-0.6) arc (-90:90:0.3 and 0.6);
			
			\draw[yellow] (-5.5,-0.6) arc (-90:90:-0.3 and 0.6);
			
			\node [circle, fill=red, inner sep = 0pt, minimum size=2mm] at (-5.5,2.25) {};
			\node [circle, fill=red, inner sep = 0pt, minimum size=2mm] at (-4,-2) {};
			\node [circle, fill=red, inner sep = 0pt, minimum size=2mm] at (-7,-0.2) {};
		\end{tikzpicture}
	\end{center}
	\caption{}
	\label{fig:connected sum}
\end{figure}

Inspired by Yu's result and Macpherson's description for toric varieties \cite{co-li-sc-11}, in this paper we show that Yu's result $H_T^*(M)\cong\mathbf{k}[M/T]$ holds for any open torus manifold with locally standard action by a different method. Our main result based on the following observations. 

\vspace{3mm}
{\bf{Observation 1:}} Let $q: M\to M/T$ be the quotient map and $M_\mrF:=q^{-1}(\mrF)$, where $\mrF$ is an open face (see Definition \ref{definition of open faces and faces for nice manifold with corners}) of $Q=M/T$. Then 
\[H^*_T(M_\mrF)\cong H^*(F)\otimes H^*(BT_F),\]
where $T_F$ is the isotropy subgroup of $M_\mrF.$ In subsection \ref{a homotopy equivalent between ET times T M mrF to}, we give the details.

\vspace{3mm}
{\bf{Observation 2:}} Let $F$ be a codimension $k$ face of $Q$, then $F$ is a connected component of $F_{i_1}\cap F_{i_2}\cap\cdots\cap F_{i_k}$ where $F_{i_j}$'s are facets of $Q$. Then for $M_\mrF$, there is a tubular neighbourhood $X_{\mrF}$ of $M_\mrF$ in $M$ such that the restriction map
\[H^*_T(X_{\mrF})\to H^*_T(X_{\mrF}\setminus M_\mrF)\]
is surjective and the kernel is the ideal of $H^*_T(X_{\mrF})$ generated by $\tau_F|_{X_{\mrF}}$. Here 
$\tau_F$ is the image of equivariant Thom class through the restriction map $H^*_T(X_{\mrF}\setminus M_\mrF)\to H^*_T(X_\mrF)$ and \red{we also call $\tau_F$ equivariant Thom class for simplifying the statements.}
In Lemma \ref{surjection of kappa F} and Lemma \ref{kernel of kappa F generated by Thom class} we give the details.

Observation 1 and Observation 2 may be well-known in toric topology but they are from our motivation that express Yu's topological face ring in terms of equivariant Thom class $\tau_F$. Although in Section \ref{topological face ring} we give an equivalent definition of Yu's topological face ring in terms of faces of $Q$ and polynomials to emphasize the combinatorial aspects of topological faces ring, we show it can also be expressed in terms of equivariant Thom class at Proposition \ref{HF[x] to HTmrF}. In particular, topological face ring $\mathbf{k}[M/T]$ is isomorphic to a subring of $\bigoplus\limits_{F\in\mcF}H_T^*(M_{\mrF})$, where $\mcF$ is the set of faces of $Q$.

\vspace{3mm}
{\bf{Observation 3:}} 
The restriction map
\[
H^*_T(M)\to\bigoplus\limits_{F\in\mcF}H_T^*(M_{\mrF})
\]
is injective. This observation is inspired by Yu's result and Macpherson's description for toric varieties. We note that each open torus manifold with locally standard action can be obtained from the following process.

Step 0: $M_0:=M_{\mathring{Q}}$, where $\mathring{Q}$ is the relative interior parts of $Q$. The isotropy subgroup for each point on $M_0$ is the identity element of $T$.

Step 1:  $M_1:=M_{0}\sqcup(M_{\mrF_1}\sqcup\cdots \sqcup M_{\mrF_m})$ where $\mrF_i$'s are all codimension one open faces of $Q$. The isotropy subgroup for each point on $M_{\mathring{F}_i}$ is the isomorphic to $S^1$. In this step we consider the case that $Q$ has codimension one faces.

Step 2: $M_2:=M_{1}\sqcup(\sqcup _{ij}M_{\mrF_{ij}})$ where $\mrF_{ij}$'s are all codimension 2 open faces of $Q$. The isotropy subgroup for each point on $M_{\mathring{F}_i}$ is the isomorphic to $S^1\times S^1$. In this step we consider the case that $Q$ has codimension two faces.

$\cdots$

After finite steps we obtain $M$. Hence we can prove Observation 3 by induction on the number of faces of $Q$, Mayer-Vietoris sequence and Observation 2. Details can be found in the proof of Proposition \ref{injective observation}.

\vspace{3mm}
{\bf{Observation 4:}} The image of the restriction map \[H^*_T(M)\to\bigoplus\limits_{F\in\mcF}H_T^*(M_{\mrF})
\] is isomorphic to the topological face ring $\mathbf{k}[M/T].$ We prove this by induction on the cardinality of $\mcF$, Mayer-Vietoris sequence and the observation that each faces of $X_{\mrF}/T$ is homotopy equivalent to $\mrF$. 

\vspace{3mm}
The paper is organized as follows. Section \ref{basic definition}, we recall some necessary definitions and properties on open torus manifolds with locally standard action and nice manifolds with corners. In Section \ref{topological face ring} we give an equivalent definition of Yu's topological face ring. In Section \ref{local observation section} we recall the higher dimensional corllar neighborhoods theorem for faces of nice manifolds with corners in toric topology. In Section \ref{An explicit formula for H*T(MmrF)}, we show that topological face ring $\mathbf{k}[M/T]$ is isomorphic to a subring of $\bigoplus\limits_{F\in\mcF}H_T^*(M_{\mrF}).$ In Section \ref{sect: equivariant cohomology of M} is \red{devoted} to show our main theorem that the equivariant cohomology ring of $M$ is isomorphic to the topological face ring $\mathbf{k}[M/T].$ In Section \ref{H(BT) algebra structure} we study the $H^*(BT)$ algebra structure of $H^*_T(M)$. In Section \ref{Sect: On equivariant characteristic classes of M}, we give an explicit formula of equivariant Stiefel-Whitney classes and Pontrjagin classes of $M$.

\bigskip
\noindent
{\bf{Acknowledgments.}}
We are grateful to Mikiya Masuda, Hiraku Abe, Tatsuya Horiguchi and Hideya Kuwata for their fruitful discussion and good comment. We also thank Li Yu for explaining the notion topological face ring to us. The first author is supported in part by NSFC: 11801186 and the second author is supported in part by NSFC: 11901218.

\bigskip

\vspace{5mm}
\section{Torus manifolds with locally standard action}\label{basic definition}\label{basic on torus manifolds with locally standard action}
In this section we recall some necessary definitions and properties on torus manifolds with locally standard action and its orbit space. More details can be found in Chapter 7 of the nice book \cite{bu-pa-15}.

\subsection{Torus manifolds with locally standard action}\label{subsection of torus manifold with locally standard action}
\begin{definition}\label{def of torus manifolds}
A torus manifold is a smooth connected closed manifold of dimension $2n$ with an effective smooth action of an $n$-dimensional torus group $T$.\footnote{Here we do note require $M^T\neq\emptyset,$ where $M^T$ is the set of fixed points of $M$ by $T$-action.}
\end{definition}
\begin{definition}
We call $T^n$-action on $\CC^n$ in the following way a standard action
\[(t_1,\dots, t_n)\cdot(z_1,\dots,z_n)=(t_1z_1,\dots, t_nz_n),\]
where $t_i, z_i\in\CC$ and $|t_i|=1$ for $1\le i\le n.$
\end{definition}
\begin{definition}\label{locally standard action}
A torus manifold $M$ is said to be locally standard if every point in $M$ has an $T$-invariant neighborhood $\mathcal{U}$ weakly equivariant diffeomorphic to an open subset $W\subseteq\CC^n$ invariant under the standard $T$-action on $\CC^n$. Here ``weakly" means that there is an automorphism $\rho :T\to T$ and a diffeomorphism $f:\mathcal{U}\to W$ such that 
\[f(ty)=\rho(t)f(y)\]
for all $t\in T, y\in\mathcal{U}.$
\end{definition}
\begin{example}
Let $S^4$ be the 4-dimensional sphere identified with the following subset of $\mathbb{C}^2\times\RR$:
\[\{(z_1,z_2,h)\in\mathbb{C}^2\times\RR: |z_1|^2|+|z_2|^2+h^2=1\}.\]
We define a $T$-action on $S^4$ as follows
\[(t_1,t_2)\cdot(z_1,z_2,h):=(t_1z_1, t_2z_2,h).\] 
Then this $T$-action on $S^4$ is locally standard. 
\end{example}
\begin{example}
Let $\mathbb{R}P^2$ be the real projective plane. Then $\mathbb{R}P^2\times T^2$ is torus manifold by the following $T^2$-action
\[(t_1,t_2)\cdot\Big([z_0:z_1:z_2],(g_1,g_2)\Big):=\Big([z_0:z_1:z_2],(t_1g_1,t_2g_2)\Big).\]
This $T$-action is locally standard. Hence a torus manifold with locally standard action may be non-orientable. 
\end{example}
\begin{example}
Toric manifolds, topological toric manifolds, quasitoric manifolds, co\"orientable toric origami manifolds, quasitoric manifolds with holes \red{and $b$-symplectic toric manifolds} are typical examples of torus manifolds with locally standard actions.
\end{example}

\begin{definition}
Let $M$ be a torus manifold with locally standard action. We call a $T$-invariant open subset of $M$ an open torus manifold with locally standard action.
\end{definition}
\begin{example}
A complex projective plane $\CC P^2$ with a $T$-action defined by 
\[(t_1,t_2)\cdot[z_0:z_1:z_2]=[z_0:t_1z_1:t_2z_2]\] is a torus manifold with locally standard action and \[\CC P^2\setminus\{[1:0:0],[0:1:0],[0:0:1]\}\] is an open torus manifold with locally standard action.
\end{example}

\vspace{3mm}
\subsection{Nice manifold with corners}
For a $2n$-dimensional torus manifold with locally standard action, the orbit space $M/T$ locally looks like some open subset of $\CC^n/T^n=\RR_{\ge 0}^n,$ so $M/T$ has a face structure from the face structure of the positive cone $\RR_{\ge 0}^n.$
\subsubsection{Face structure on $\RR_{\ge 0}^n$} Set 
$\RR_{\ge 0}^n=\{(x_1,\dots, x_n)\in\RR^n: x_i\ge 0~\text{for}~i=1,\dots, n\}.$
\begin{definition}
The codimension $c(x)$ of $x\in\RR^n_{\ge 0}$ is the number of zero coordinates of $x$.
\end{definition}
\begin{example}
For $n=2$, we have $c\Big((0,0)\Big)=2$, $c\Big((0,2)\Big)=1$ and $c\Big((2,3)\Big)=0.$
\end{example}
\begin{definition}
We call a connected component of $c^{-1}(k)$ an open face, of codimension $k$, of $\RR_{\ge 0}^n$. A closed face (or simply face) of $\RR_{\ge 0}^n$ is the closure of an open face.
\end{definition}
\begin{example}
For $n=2$, $\{(x_1,0):x_1>0\}$ and $\{(0,x_2):x_2>0\}$ are codimension $1$ open faces of  $\RR_{\ge 0}^2$. $\{(x_1,0):x_1\ge 0\}$ is a codimension $1$ face of $\RR_{\ge 0}^2$.
\end{example}

\vspace{5mm}
\subsubsection{Face structure on manifolds with corners}
\begin{definition}\label{manifolds with corners}
A manifold with corners (of dimension $n$) is a topological manifold $Q$ with boundary together ($\partial Q$ can be an empty set) 
with an atlas $(U_i, \varphi_i)$ consisting of homeomorphisms $\varphi_i: U_i\to W_i$ onto open subsets $W_i\subseteq\RR^n_{\ge 0}$
such that $\varphi_i\varphi_j^{-1}:\varphi_j(U_i\cap U_j)\to\varphi_i(U_i\cap U_j)$
 is a diffeomorphism for all $i,j$. (A homeomorphism between open
subsets in $\RR^n_{\ge 0}$ is called a diffeomorphism if it can be obtained by restriction of a
diffeomorphism of open subsets in 
 $\RR^n$.)
\end{definition}
Note that for $q\in Q$, its codimension $c(q)$ is well-defined, which is defined by $c\Big((\varphi_i(q)\Big)$, where $(U_i,\varphi_i)$ is a local chart of $q$.
\begin{definition}\label{definition of open faces and faces for nice manifold with corners}
Let $Q$ be $n$-dimensional manifold with corners. An open face of $Q$ of
codimension $k$ is a connected component of $c^{-1}(k)$. A closed face (or simply face)
of $Q$ is the closure of an open face. A facet is a face of codimension $1$.
\end{definition}

\begin{definition}\label{nice manifold with corners}
A manifold with corner is nice if every codimension $k$ face is contained in exactly $k$ facets.
\end{definition}

\begin{proposition}[\cite{bu-pa-15}, Proposition 7.4.13]
Let $M$ be an open torus manifold with locally standard action, then the orbit space $M/T$ is a nice manifold with corners.
\end{proposition}
\begin{example}
The orbit space $Q$ of $S^4$ with the torus action given in subsection \ref{subsection of torus manifold with locally standard action} is a nice manifold with corners (See figure \ref{fig:cycle}). Here $F_1$ and $F_2$ are facets of $Q$ and $p$ and $q$ are $0$-dimension face of $Q$. $F_1\setminus\{p,q\}$ and $Q\setminus (F_1\cup F_2)$ are open faces of $Q$.

	\begin{figure}[h]
	\begin{center}
		\begin{tikzpicture}[scale=0.3]
			\pgfsetfillopacity{1}
			\coordinate[label = left: $F_1$] (A) at (-5,0);
			\coordinate[label = right: $F_2$] (A) at (5,0);
			\filldraw[fill=yellow!50, draw=blue, line width=0.5 mm] (0,-4) .. controls (-4,-1) and (-4,1) .. (0,4)
			.. controls (4,1) and (4,-1) .. (0,-4);
		    \node [circle, fill=red, inner sep = 0pt, minimum size=2mm] (p) at (0,4) {};
			\node [circle, fill=red, inner sep = 0pt, minimum size=2mm] (q) at (0,-4) {};
			\node [red,above] at (p.north) {p};
			\node [red,below] at (q.south) {q};	
			\draw [->,line width=0.3 mm] (-5.3,0) to (-3.3,0);	
			\draw [<-,line width=0.3 mm] (3.3,0) to (5.3,0);	
		\end{tikzpicture}
	\end{center}
	\caption{}
	\label{fig:cycle}
\end{figure}
\end{example}

\subsubsection{On characteristic submanifolds}
\begin{definition}
Let $M$ be a torus manifold with locally standard action. A connected codimension-two submanifold of $M$ is called characteristic if it is a connected component of the set fixed pointwise by a certain circle subgroup of $T$.
\end{definition}
\begin{example}
Let $T$ acts on $S^4$ as in subsection \ref{subsection of torus manifold with locally standard action}. Then 
$\{(z_1,z_2,h)\in S^4: z_1=0\}$ and $\{(z_1,z_2,h)\in S^4: z_2=0\}$ are characteristic submanifolds of $S^4.$
\end{example}
Since a torus manifold $M$ is compact, the number of characteristic submanifolds of $M$ is finite, and so is the number of characteristic submanifolds of an open torus manifold with locally standard action.
It is not difficult to obtain the following facts from locally standard action:
\begin{itemize}
\item
If $M_i$ and $M_j$ be two distinct characteristic submanifolds of $M$ then they intersects transversally;
\item 
The preimage $M_F=q^{-1}(F)$ of a codimension $k$ face $F\subseteq Q$ is a codimension $2k$ closed $T$-invariant submanifold of $M$ and $M_F$ is a connected component of an intersection of $k$ characteristic submanifolds, where $q: M\to Q$ is the quotient map.
\end{itemize}

The following result is well-known in toric topology, but for convenience for the readers we give a proof.
\begin{lemma}
Let $M$ be a torus manifold with locally standard action and $M_i$ be a characteristic submanifold of $M$. Then the normal bundle of $M_i$ in $M$ is a complex line bundle over $M_i$.
\end{lemma}
\begin{proof}
Let $\nu_i$ be the normal bundle of $M_i$ in $M$ and $M_i$ is fixed by a circle subgroup $\lambda_i(S^1)$ of $T$. Since $T$ acts on $M$ effectively, $\lambda_i(S^1)$ acts on each fiber of $\nu_i$ effectively. For any $v\in\nu_i$, we define $Jv :=\sqrt{-1}\cdot v.$ Hence $\nu_i$ is a complex line bundle over $M_i.$ 
\end{proof}
For each characteristic submanifold $M_i$ of $M$, we fix a complex structure on $\nu_i$. This determines an element $\phi_i$ in $H^2_T(M,M_i)$ by equivariant Thom isomorphism. From now on we denote the image of $\phi_i$ in $H^2_T(M)$ by $\tau_i$ through the restriction map \[H^2_T(M,M_i)\to H^2_T(M).\]

\vspace{5mm}
\section{Topological face ring}\label{topological face ring}
The notion topological face ring were introduced by L. Yu in \cite{yu-21}. In this section we give an equivalent definition of topological face ring and recall some properties on topological face rings.

Let $Q$ be a nice manifold with corners. We recall some notations 
\begin{itemize}
\item $\mcF :=\{\text{faces of}~Q\}$;
\item $\mathring{\mcF} :=\{\text{open faces of} ~Q\}$.
\end{itemize}
Let $\{F_1,\dots, F_m\}$ be the set of facets of $Q$. Since $Q$ is a nice manifold with corners, for each codimension $k$-face $F$ of $Q$, where $k\ge 1$, $F$ is a connected component of $F_{i_1}\cap\cdots\cap F_{i _k}$. Hence we can associated each face of $Q$ to a subset of $[m].$ Namely we have a map
\[
\begin{split}
\Psi: \mcF&\to 2^{m}\\
F&\mapsto\{i_1,\dots, i_k\}.
\end{split}
\]
If $E$ is a face of $F$, we have $\Psi(F) \subseteq \Psi(E).$

Let 
\[
A:=\bigoplus\limits_{F\in\mcF} H^*(F)[\mathbf{x}_F]
\]
where $H^*(F)[\mathbf{x}_F]$ means \red{$H^*(F;\mathbf{k})\bigotimes\limits_{\mathbf{k}}\mathbf{k}[x_{i_1},\dots, x_{i_k}]$} and $\{i_1,\dots, i_k\}=\Psi(F).$ For each element $\alpha\in A$, we denote the $F$-component of $\alpha$ by $\alpha_F$.  Let $\alpha_F=\sum\limits_{\mathbf{m}}a_{\mathbf{m},F}(\mathbf{x}_F)^{\mathbf{m}}$, where $a_{\mathbf{m},F}\in H^*(F)$, \red{$\mathbf{m}=(m_1,\dots, m_k)\in \Z_{\ge0}^k$}, $(\mathbf{x}_F)^{\mathbf{m}}=x_{i_1}^{m_1}\dots x_{i_k}^{m_k}$ and \red{$\deg(\alpha_F):=\deg(a_{\mathbf{m},F})+2(m_1+\dots+m_k).$} For an ordered pair of faces $(F,E)$, we  define a morphism $\phi_{FE}: H^*(F)[\mathbf{x}_F]\to H^*(E)[{}{\mathbf{x}_E}]$ as follows.
\begin{equation}\label{def of phiFE}
\phi_{FE}\left(\sum\limits_{\mathbf{m}}a_{\mathbf{m},F}(\mathbf{x}_F)^{\mathbf{m}}\right)=
\begin{cases}
\sum\limits_{\mathbf{m}}a_{\mathbf{m},F}|_E(\mathbf{x}_F)^{\mathbf{m}}\quad\text{if}~E\subseteq F\\
0\hspace{34mm}\text{otherwise}
\end{cases}
\end{equation}
where $a_{\mathbf{m},F}|_E$ is the restriction of $a_{\mathbf{m},F}$ on $E$ induced by the inclusion map $E\subseteq F.$

\begin{definition}\label{def of F-face elements}
Let $F$ be a face of $Q$ and $\Psi(F)=\{i_1,\dots, i_k\}$. An element $\alpha\in A$ is called an $F$-face element if the following two conditions hold:
\begin{itemize}
\item
$\alpha_F=\sum\limits_{\mathbf{m}}a_{\mathbf{m},F}(\mathbf{x}_F)^{\mathbf{m}}$ satisfying $\mathbf{m}\in\N_{>0}^k;$
\item
$\alpha_E=\phi_{FE}(\alpha_F).$
\end{itemize}
\end{definition}
\begin{definition}\label{topological face module}
The $\mathbf{k}$-submodule generated by the faces elements of $A$ is called a topological face $\mathbf{k}$-module, denoted by $\mathbf{k}[Q].$ Namely
\[
\mathbf{k}[Q]=\text{Span}_{\mathbf{k}}\langle\alpha: \alpha~\text{is a $F$-face element of $A$ for some $F\in\mcF$}\rangle.
\]
\end{definition}
\begin{lemma}\label{product of F-face elements}
Let $\alpha$ and $\alpha'$ be $F$-face elements, then $\alpha\alpha'$ is also an $F$-face element.
\end{lemma}
\begin{proof}
This lemma follows from Definition \ref{def of F-face elements} and the observation that $\phi_{FE}$ is a ring homomorphism for any ordered pair $(F,E)$.
\end{proof}
Let $(E,G)$ be an ordered pair, then for any $E$-face element $\alpha$, we can define an element  $\Theta_{EG}(\alpha)$ of $A$ as follows: for $F\in\mcF$ 
\begin{equation}\label{def of ThetaEG}
\Big(\Theta_{EG}(\alpha)\Big)_F=\phi_{GF}\Big(\phi_{EG}(\alpha_E)\Big).
\end{equation}
\begin{lemma}\label{from E-face element to G-face element}
Let $(E,G)$ be an ordered pair and $\alpha$ be an $E$-face element, then $\Theta_{EG}(\alpha)$ is a $G$-face element.
\end{lemma}
\begin{proof}
{\bf{Case 1}:} $G\nsubseteq E$. Then $\phi_{EG}(\alpha_E)=0$ by (\ref{def of phiFE}). Hence by (\ref{def of ThetaEG}), for each $F\in\mcF$, we have $\Big(\Theta_{EG}(\alpha)\Big)_F=0,$ which means that $\Theta_{EG}(\alpha)=0$. In particular, $\Theta_{EG}(\alpha)$ is a $G$-face element.

{\bf{Case 2}:} $G\subseteq E$. Then 
\[\Big(\Theta_{EG}(\alpha)\Big)_G=\phi_{EG}(\alpha_E)\]
by (\ref{def of ThetaEG}), 
which means that $\Theta_{EG}(\alpha)$ satisfies the first condition of Definition \ref{def of F-face elements} for $G$-face elements. Let $F\in\mcF$. By (\ref{def of ThetaEG}) and the last equality, we have 
\[\Theta_{EG}(\alpha)_F=\phi_{GF}\Big(\Theta_{EG}(\alpha)_G\Big)\]
which means that $\Theta_{EG}(\alpha)$  satisfies the second condition of Definition \ref{def of F-face elements} for $G$-face elements. Therefore $\Theta_{EG}(\alpha)$ is a $G$-face element.
\end{proof}
\begin{proposition}
Let $Q$ be a nice manifold with corners. Then its topological face $\mathbf{k}$-module
$\mathbf{k}[Q]$ is a subring of $A$.
\end{proposition}
\begin{proof}
It is enough to show that if $\alpha$ and $\beta$ are $E_1$-face element and $E_2$-face element of $A$ respectively, then 
$\alpha\beta\in\mathbf{k}[Q].$

{\bf{Case 1}:} $E_1\cap E_2=\emptyset.$ Let $F\in\mcF$, then $F\nsubseteq E_1\cap E_2$. Hence $F\nsubseteq E_1$ or $F\nsubseteq E_2$ which means $\phi_{E_1F}(\alpha_{E_1})=0$ or $\phi_{E_2F}(\beta_{E_2})=0$. Therefore we have
 \[
 \begin{split}
 (\alpha\beta)_F&=\alpha_F\beta_F\\
 &=\phi_{E_1F}(\alpha_{E_1})\phi_{E_2F}(\beta_{E_2})\\
 &=0.
 \end{split}
 \]
 This implies that $\alpha\beta=0\in\mathbf{k}[Q]$ for this case.

{\bf{Case 2}:} $E_1\cap E_2\ne\emptyset.$ Then $E_1\cap E_2=G_1\sqcup\cdots \sqcup G_h$ where $G_i$'s are faces of $Q.$ For each $i\in\{1,\dots, h\}$, $\Theta_{E_1G_i}(\alpha)$ and $\Theta_{E_2G_i}(\beta)$ are $G_i$-face elements by Lemma \ref{from E-face element to G-face element}, so 
$
\Theta_{E_1G_i}(\alpha)\Theta_{E_2G_i}(\beta)$
is also a $G_i$-face element by Lemma \ref{product of F-face elements}.  We claim that 
\[\alpha\beta=\sum\limits_{i=1}^h\Theta_{E_1G_i}(\alpha)\Theta_{E_2G_i}(\beta).\]
It suffices to show that 
\begin{equation}\label{alpha beta F=sum part}
(\alpha\beta)_F=\sum\limits_{i=1}^h(\Theta_{E_1G_i}(\alpha)\Theta_{E_2G_i}(\beta))_F
\end{equation}
holds for any $F\in\mcF.$

 {\bf{Case 2-1}:} $F\nsubseteq E_1\cap E_2$. By the same argument for Case 1, we have 
\[(\alpha\beta)_F=0.\]
Since $F\nsubseteq E_1\cap E_2$, for any $i\in\{1,\dots, h\}$ we have $F\nsubseteq G_i$.
Hence the $F$-component of the $G_i$-face element $\Big(\Theta_{E_1G_i}(\alpha)\Theta_{E_2G_i}(\beta)\Big)$ is $0$ by (\ref{def of phiFE}). Therefore we have 
\[\sum\limits_{i=1}^h\Big(\Theta_{E_1G_i}(\alpha)\Theta_{E_2G_i}(\beta)\Big)_F=0.\]
Therefore (\ref{alpha beta F=sum part}) holds for this case.

{\bf{Case 2-2}:} $F\subseteq E_1\cap E_2$. Then there exists $i_0\in\{1,\dots, h\}$ such that $F\subseteq G_{i_0}.$ Hence we have
\[\begin{split}
(\alpha\beta)_F&=\alpha_F\beta_F\\
&=\phi_{E_1F}(\alpha_{E_1})\phi_{E_2F}(\beta_{E_2})\\
&=\Big(\phi_{G_{i_0}F}\circ\phi_{E_1G_{i_0}}(\alpha_{E_1})\Big)\Big(\phi_{G_{i_0}F}\circ\phi_{E_2G_{i_0}}(\beta_{E_2})\Big)\\
&=\Big(\Theta_{E_1G_{i_0}}(\alpha)\Big)_F\Big(\Theta_{E_2G_{i_0}}(\beta)\Big)_F\\
&=\Big(\Theta_{E_1G_{i_0}}(\alpha)\Theta_{E_2G_{i_0}}(\beta)\Big)_F.
\end{split}\]
Since $G_i\cap G_{i_0}=\emptyset$ for   $i\in\{1,\dots,h\}\setminus\{i_0\},$ we have $F\nsubseteq G_i$. Hence for $i\ne i_0$, we have 
\[\Big(\Theta_{E_1G_{i}}(\alpha)\Theta_{E_2G_{i}}(\beta)\Big)_F=0.\]
Therefore (\ref{alpha beta F=sum part}) also holds for this case. Hence the claim holds which implies that $\alpha\beta\in\mathbf{k}[Q]$ for Case 2. 
\end{proof}
\begin{definition}
We call $\mathbf{k}[Q]$, defined in Definition \ref{topological face module}, a topological face ring of $Q$.
\end{definition}

\vspace{3mm}
\section{\red{Higher dimensional collar
neighborhoods of faces}}\label{local observation section}
In this section we give the details of higher dimensional corllar neighborhoods theorem for faces of nice manifolds with corners. The results in this section may be well-known, but for the reader's convenience we give proofs.

Let $M$ be a $2n$-dimensional torus manifold with a locally standard action. Let $Q=M/T$ be the orbit space of $M$. Then $Q$ is a nice manifold with corners and the faces structure is determined by the orbit type of torus action on $M$ (see \cite[Proposition 7.4.13]{bu-pa-15}). The followings are the notations which we needed in this section.
\begin{itemize}
\item $q: M\to Q$ be quotient map. 
\item $\mcF$ is the set of faces of $Q$.
\item $\mcrF$ is the set of open faces of $Q$.
\item {$\mrF$} is an open face which is the relative interior of a face $F$.
\item For $F\in\mcF$, $M_F :=q^{-1}(F)$ and $M_{\mrF}:=q^{-1}(\mrF)$
\end{itemize}

\begin{lemma}\label{T-invariant open tubular nbhd of MF}\cite[Theorem 2.2 in Chapter VI]{bre-72}
\red{For $F\in\mcF$, there} exists a $T$-invariant open tubular neighbourhood $X_F$ of $M_F$ in $M$ such that $X_F$ is $T$-equivariant diffeomorphic to the $T$-equivariant normal bundle $\nu_F$ of the embedding $M_F\subset M.$
\end{lemma}

Hence from now on we can identify the $T$-equivariant normal bundle $\nu_F$ with a $T$-invariant tubular neighbourhood $X_F$ of $M_F$ in $M$. Let $\pi_F: X_F\to M_F$ be the projection induced by the projection map of the normal bundle $\nu_F \rightarrow M_{F}$ and \begin{equation}\label{X mathring F}
X_{\mrF}:=\pi_{F}^{-1}(M_{\mrF}). 
\end{equation}
Since $\mrF$ is a relative interior part of $F$, $X_\mrF$ is a $T$-invariant open subset of $X_F$, which means that $X_\mrF$ is an open torus manifold with locally standard action. Therefore the orbit space $X_{\mrF}/T$ is a nice manifold with corners (see \cite{bu-pa-15}). We observe the following useful fact which will be used in the proof of our main theorem. In the remaining part of this section, we devote to show the following proposition.
\begin{proposition}\label{one important observation}
Each face of $X_{\mrF}/T$ is homotopy equivalent to $F$.
\end{proposition}
\begin{lemma}\label{X mathring F T-equivariant diff to nu mathring F}
The open torus manifold $X_{\mrF}$ is $T$-equivariant diffeomorphic to $\nu_{\mrF}:=\nu_{F}|_{M_\mrF}$. 
\end{lemma}
\begin{proof}
	This is easily check directly from the definition of $X_{\mrF}$ and $\pi_F: X_F\to M_F$.
	\end{proof}
This lemma shows that $\nu_{\mrF}$ is an open torus manifold with locally standard action and
\begin{equation}
X_{\mrF}/T\cong\nu_{\mrF}/T
\end{equation}
as nice manifolds with corners. 
\begin{lemma}\label{local model for facet}
If $F$ is a facet of $Q=M/T$, then 
\begin{enumerate}
\item $\nu_F/T$ is homeomorphic to $F\times\RR_{\ge 0}$ as manifolds with corners.
\item each face of $\nu_{\mrF}/T$ is homotopy equivalent to $F$.
\end{enumerate}
\end{lemma}
\begin{proof}
\begin{enumerate}
\item
Since $F$ is a facet of $Q=M/T$, $M_F$ is a characteristic submanifold of $M$. Hence there is a subgroup $H$ of $T$ such that $H$ fixes $M_{F}$ pointwise and $H$ is isomorphic to $S^1$. Let $\nu_F\xrightarrow{p_F}M_F$ be the normal bundle of the embedding $M_F\subset M$ and $S(\nu_F)$ is the unit circle bundle of $\nu_F$ with respect to a $T$-invariant Riemann metric $g$ on $\nu_F.$ Then the following diagram is commutative,
\begin{equation}\label{S nu F to MF and nu F to MF}
	\begin{CD}
		S(\nu_{F}) @>{p_F}>> M_{F}\\
		@V{\iota_F}VV @VV{id}V\\
		\nu_{F} @>{p_F}>> M_{F}
	\end{CD}
\end{equation}
where $\iota_F$ is the natural inclusion map,  $id$ is the identity map and all maps in this diagram are $T$-equivariant. Since $T$ acts on $M$ effectively and $H$ acts on $M_F$ trivially, $H$ acts on the fiber direction of $\nu_F$ faithfully. Therefore $H$ acts on the fiber of $S(\nu_{F})$ freely. Therefore $\overline{p_F}: S(\nu_{F})/H\to M_F$ is a $T/H$-equivariant homeomorphsim, where $\overline{p_F}$ is induced by $p_F$ in (\ref{S nu F to MF and nu F to MF}). Hence there is a $T/H$-equivariant homeomorphism $s^\prime: M_F\to S(\nu_{F})/H$ which is the inverse of $\overline{p_F}$. Note that the following diagram is commutative,
\[\begin{tikzcd}[column sep=small]
S(\nu_{F})/H \arrow{r}{\overline{\iota_F}}  \arrow[swap]{rd}{\overline{p_F}} 
& \nu_{F}/H \arrow{d}{\overline{p_F}} \\
& M_F
\end{tikzcd}
\]
where $\overline{\iota_F}$ and $\overline{p_F}$ are induced by $\iota_F$ and $p_F$ in (\ref{S nu F to MF and nu F to MF}) respectively. Let $s: M_F\to \nu_F/H$ be the composition map 
\[M_F\xrightarrow{s^\prime}S(\nu_{F})/H\xrightarrow{\overline{\iota_F}}\nu_F/H.\]
Hence for each $x\in M_F$, $s(x)\ne \overline{0},$ where $\overline{0}\in(\nu_F/H)_x=\overline{p_F}^{-1}(x)$ is the image of $0\in(\nu_F)_x=p_F^{-1}(x)$. Therefore we have the following $T/H$-equivariant map
\begin{equation}\label{nuF/H to MF times R nonnegative}
\begin{split}
\psi: \nu_F/H&\to M_F\times\RR_{\ge 0}\\
(x,\overline{v})&\mapsto(x,\frac{||\overline{v}||}{||s(x)||}),
\end{split}
\end{equation}
where $T/H$ acts on $\RR_{\ge 0}$ trivially, $||\overline{v}||:=g(v,v)$. Since $g$ is a $T$-invariant Riemann metric on $\nu_F$, $\psi$ is well-defined. Also we have the following $T/H$-equivariant map
\[
\begin{split}
\Gamma: M_F\times\RR_{\ge 0}&\to \nu_F/H\\
(x,t)&\mapsto(x,ts(x)).
\end{split}
\]
We check directly that $\psi=\Gamma^{-1}.$ Hence $M_F\times\RR_{\ge 0}$ and $\nu_F/H$ have the same orbit type. Therefore $\nu_F/T=(\nu_F/H)/(T/H)$ is homeomorphic to \red{$(M_F\times\RR_{\ge 0})/(T/H)=\Big(M_F/(T/H)\Big)\times \RR_{\ge 0}=F\times\RR_{\ge 0}$} as nice manifolds with corners.
\item The above argument imply that $\nu_{\mrF}/T$ is homeomorphic to $(M_{\mrF}\times\RR_{\ge 0})/(T/H)=M_{\mrF}/(T/H)\times \RR_{\ge 0}=\mrF\times\RR_{\ge 0}$ as nice manifolds with corners. Therefore each face of $\nu_{\mrF}/T$ is homotopy equivalent to $\mrF$. Then this lemma follows from the fact that $\mrF$ is homotopy equivalent to $F$ through the inclusion map $\mrF\subseteq F.$
\end{enumerate}
\end{proof}

\begin{corollary}\label{local model for faces}
Let $F$ be a codimension $k$ face of $Q=M/T$, then 
\begin{enumerate}
\item
$\nu_F/T$ is homeomorphic to $F\times\mathbb{R}^k_{\ge 0}$ as manifolds with corners.
\item each face $E$ of $\nu_{\mrF}/T$ is homeomorphic to $\mrF \times \mathbb{R}^{i}_{\ge 0}$, where $i=\text{dim}(E)-\text{dim}(F)$ and $\mathbb{R}^{i}_{\ge 0}$ is a dimension $i$ face of $\mathbb{R}^{k}_{\ge 0}$.
\end{enumerate}
\end{corollary}
\begin{proof}
\begin{enumerate}
\item Since $F$ is a codimension $k$ face of $Q$, $F$ is a connected component of $F_{i_1}\cap F_{i_2}\cap\cdots\cap F_{i_k},$ where $F_{i_j}$'s for $1\le j\le k$ are \red{facets} of $Q$. Hence $M_F=q^{-1}(F)$ is a connected component of $M_{F_{i_1}}\cap M_{F_{i_2}}\cap\cdots\cap M_{F_{i_k}}$ where $M_{F_{i_j}}=q^{-1}(F_{i_j}).$ Let $g$ be a $T$-invariant Riemann metric on $M$. Then $TM|_{M_{F_{i_j}}}\cong TM{_{F_{i_j}}}\oplus\nu_{F_{i_j}}$ as $T$-invariant vector bundles and $\nu_{F_{i_j}}\cong (TM{_{F_{i_j}}})^\bot$ as $T$-invariant vector  bundles. Since for each $x\in M_F$, $M_{F_{i_1}},\dots, M_{F_{i_k}}$ intersects at $x$ transversely, we have 
\begin{equation}\label{split of nu F}
\nu_F\cong\oplus_{j=1}^k\nu_{F_{i_j}}|_{M_F}
\end{equation}
as $T$-invariant vector bundles. \red{In addition we have the following equivariant bundle map}
\begin{equation}\label{nu F to nu Fi1 times dots nu Fik}
	\begin{CD}
		\nu_{F} @>>> \nu_{F_{i_1}} \times \cdots \times \nu_{F_{i_k}}\\
		@V{p_{_F}}VV @VV{p_{_{F(i_1,\dots,i_k)}}}V\\
		M_{F} @>{\triangle}>> M_{F_{i_1}}\times \cdots \times M_{F_{i_k}},
	\end{CD}
\end{equation}
where $\triangle(x)=(x,\cdots,x)\in M_{F_{i_1}}\times \cdots\times M_{F_{i_k}}$ for any $x \in M_{F}$ and $p_{_{F(i_1,\dots,i_k)}}=p_{_{F_{i_1}}}\times\cdots \times p_{_{F_{i_k}}}.$
Let $T_{F_{i_j}}$ be the subgroup of $T$ such that $T_{F_{i_j}}$ fixes $M_{F_{i_j}}$ pointwise and $T_{F_{i_j}}\cong S^1.$ Hence $M_F$ is a connected component of $M^{T_F}$, where $T_F=T_{i_1}\times\cdots\times T_{i_k}$ and $M^{T_F}=\{x\in M: t\cdot x=x, \forall t\in T_F\}.$
 Therefore (\ref{nu F to nu Fi1 times dots nu Fik}) induces the following commutative diagram.
\begin{equation}
	\begin{CD}
		\nu_{F}/T_F @>>> \nu_{F_{i_1}}/T_{i_1} \times \cdots \times \nu_{F_{i_k}}/T_{i_k}\\
		@V{\overline{p_{_F}}}VV @VV{\overline{p_{_{F(i_1,\dots,i_k)}}}}V\\
		M_{F} @>{\triangle}>> M_{F_{i_1}}\times \cdots \times M_{F_{i_k}},
	\end{CD}
\end{equation}
By (\ref{nuF/H to MF times R nonnegative}) in the proof of Lemma \ref{local model for facet}, we can identify $\overline{p_{_{F_{i_j}}}}: \nu_{F_{i_j}}/T_{i_j}\to M_{F_{i_j}}$ with the \red{first projection} $M_{F_{i_j}}\times\RR_{\ge 0}\to M_{F_{i_j}}.$ Hence we can identify $\overline{p_{_F}}:\nu_F/T_F\to M_F$ with the \red{first projection}
$M_F\times\RR^k_{\ge 0}\to M_F.$ In particular, We have $\nu_{F}/T_F$ is equivariantly homeomorphic to $M_F \times \mathbb{R}_{\geq 0}^k$ as \red{$T/T_F$- manifolds}, where $T/T_F$ acts on $R_{\ge 0}^k$ trivially. Hence \red{$M_F\times\RR_{\ge 0}^k$} and \red{$\nu_F/T_F$} have the same orbit type. Therefore $\nu_F/T=(\nu_F/T_F)/(T/T_F)$ is homeomorphic to \red{$(M_F\times\RR^k_{\ge 0})/(T/T_F)=\Big(M_F/(T/T_F)\Big)\times \RR^k_{\ge 0}=F\times\RR^k_{\ge 0}$} as nice manifolds with corners. 
\item The above argument imply that $\nu_{\mrF}/T$ is homeomorphic to $(M_{\mrF}\times\RR^k_{\ge 0})/(T/H)=M_{\mrF}/(T/H)\times \RR^k_{\ge 0}=\mrF\times\RR^k_{\ge 0}$ as nice manifolds with corners. By the face structure of $\mrF\times\RR^k_{\ge 0}$, the conclusion holds. 

\end{enumerate}
\end{proof}

\section{Topological face rings and equivariant Thom classes}\label{An explicit formula for H*T(MmrF)}
\red{In Section \ref{topological face ring} we expressed} Yu's topological face ring in terms of orbit spaces and polynomials, but it is not obvious to see the meaning of the variables $x_1,\dots, x_m$. In this section we relate $x_i$ to equivariant Thom classes 
$\tau_i$ which is defined at the end of Section \ref{basic definition}. As a conclusion of this section, we show that the topological face ring $\mathbf{k}[M/T]$ identifies with a subring of $\bigoplus\limits_{F\in\mcF}H_T^*(M_{\mrF})$, where $\mcF$ is the set of faces of $Q$.%

\vspace{3mm}
\subsection{On $H^*_{T_F}(p)$}\label{subsection: On HTF(p)}
Let $\{F_1,\cdots, F_m\}$ be the set of \red{facets} of $Q=M/T$ and $F$ be a codimension $k$ face of $Q$. Since $Q$ is a nice manifold with corners, $F$ is a connected component of $F_{i_1}\cap\cdots\cap F_{i_k},$ where $F_{i_j}$'s are facets of $Q$. Let $T_{i}$ be the subgroup of $T$ such that $T_{i}$ fixes $M_{F_i}$ \red{pointwise} and $T_{F_{i}}\cong S^1$, where $M_{F_i}=q^{-1}(F_i)$. Let $T_F:=T_{i_1}\times\cdots\times T_{i_k}.$ Since $T$ acs on $M$ locally standard, for any $p\in M_{\mrF}$ the isotropy subgroup of $p$ is $T_F$. We denote $\tau_i|_{M_{\mrF}}$ by $\tau_i(M_\mrF)$, which is an element of  $H^2_T(M_{\mrF})$, for $i=1,2\dots,m.$ For $p\in M_{\mrF}$, we denote $\text{Res}_{T_F}(\tau_i)|_p$ by $\Big(\text{Res}_{T_F}(\tau_i)\Big)(p)$  where $\text{Res}_{T_F}$ denotes the restriction map from $H^*_T(X)$ to $H^*_{T_F}(X)$ for any topological sapce $X$ with a $T$-action. We  identify $H^2_{T_F}(p)=H^2(BT_F)$ with $\mathrm{Hom}(T_F, S^1)$. Since $T_F=T_{{i_1}}\times\cdots\times T_{{i_k}}$, $\mathrm{Hom}(T_F, S^1)\cong \mathrm{Hom}(T_{{i_1}},S^1)\oplus\cdots\oplus \mathrm{Hom}(T_{{i_k}},S^1)$ which can be identified with $H^2(BT_{{i_1}})\oplus\cdots\oplus H^2(BT_{{i_k}}).$ By (\ref{split of nu F}), we have
\red{\begin{equation}\label{vector bundle decomposition}
\nu_F\cong\oplus_{j=1}^k(\nu_{F_{i_j}}|_{M_F})
\end{equation}}
It is well-known that the $T_{{i_j}}$-equivariant  Euler class of $\nu_{F_{i_j}}|_p$ is $\Big(\text{Res}_{T_{i_j}}(\tau_{i_j})\Big)(p).$ \red{Since $T$ acts on $M$ effectively, $\Big(\text{Res}_{T_{i_j}}(\tau_{i_j})\Big)(p)$ is a $\bf{k}$-basis of $H_{T_{{i_j}}}^2(p)$.} This implies that 
\[\Big(\text{Res}_{T_{i_1}}(\tau_{i_1})\Big)(p),\cdots,\Big(\text{Res}_{T_{i_k}}(\tau_{i_k})\Big)(p)\] is a $\mathbf{k}$-basis of $H^2_{T_F}(p).$ Since $H^*_{T_F}(p)$ is a polynomial ring generated by degree $H^2_{T_F}(p)$, we have
\[H^*_{T_F}(p)=\mathbf{k}\bigg[\Big(\text{Res}_{T_{i_1}}(\tau_{i_1})\Big)(p),\dots,\Big(\text{Res}_{T_{i_k}}(\tau_{i_k})\Big)(p)\bigg].\]

\vspace{3mm}
\subsection{A homotopy equivalence between $ET\times_T M_{\mrF}$ and $(BT_F\times\{p\})\times F$}\label{a homotopy equivalent between ET times T M mrF to}
Since $T$ acts on $M$ \red{locally standard}, the following exact sequence is splitting.
\[1\to T_F\to T\to T/T_F\to 1,\]
where $T_F\to T$ is the inclusion map and $T\to T/T_F$ is the quotient map. Hence we have the following isomorphism:
\[f: T\cong T_F\times T/T_F.\]
Note that for any $p\in M_{\mrF}$ the isotropy subgroup of $p$ is $T_F$, so $T/T_F$ acts freely on $M_{\mrF}.$ Therefore, $ET\times_T M_{\mrF}$ is homotopy equivalent to $BT_F\times \mrF.$ In the following we explain this homotopy equivalent process. Let $p\in M_{\mrF}$. Then the map
\[\begin{split}
\sigma_{p}:M_{\mrF}&\to\{p\}\times M_{\mrF}\\
x&\mapsto (p,x)
\end{split}\]
is a $T$-equivariant diffeomorphism, where $T$ acts on $\{p\}\times M_{\mrF}$ through the above isomorphism $T\xrightarrow{f}T_F\times T/T_F.$
Therefore we have
\[\begin{split}
ET\times_T M_{\mrF}&\sim ET\times_T(\{p\}\times M_{\mrF})\\
&\sim(ET_F\times_{T_F}\times\{p\})\times(E(T/T_F)\times_{T/T_F}M_{\mrF})\\
&\sim (BT_F\times\{p\})\times\mrF\\
&\sim (BT_F\times\{p\})\times F
\end{split}\]
and we denote this homotopy equivalence by \red{$\psi_p: ET\times_T M_{\mrF} \xrightarrow{\hspace{2mm}\sim\hspace{2mm}} (BT_F\times\{p\})\times F.$}
From the above homotopy equivalent process, we can obtain following commutative diagram.
\begin{equation}\label{Psi F,p star and qF}
	\begin{CD}
		ET\times_T M_{\mrF} @>{\psi_{p}}>> (BT_F\times\{p\})\times F\\
		@V{\overline{q}}VV @VV{p_2}V\\
		F @>{\text{identity~map}}>> F
	\end{CD}
\end{equation}
where $\overline{q}$ is induced by the composition map \red{$ET\times_TM_\mrF\to\mrF \to F$} from the second projection and inclusion maps and $p_2$ is the second projection map. Hence for any $\alpha\in H^*(F),$ we have \begin{equation}\label{psi inverse widetilde qf}
(\psi_{p}^*)^{-1}\overline{q}^*(\alpha)=p_2^*(\alpha).
\end{equation}
\begin{lemma}\label{psi inverse tau ij F} 
Let $F$ be a codimension $k$ face of $Q=M/T$ and $F$ is a connected component of $F_{i_1}\cap\cdots\cap F_{i_k},$ where $F_{i_j}$'s are facets of $Q$.
For any $1\le j\le k$, we have
\[\Big((\psi_{p}^*)^{-1}(\tau_{i_j}(M_\mrF)\Big)=p_1^*\bigg(\Big(\text{Res}_{T_{i_j}}(\tau_{i_j})\Big)(p)\bigg)+p_2^*(\alpha_j),\]
where $p_1: (BT_F\times\{p\})\times F\to (BT_F\times\{p\})$ is the first projection and $\alpha_j\in H^2(F).$
\end{lemma}
\begin{proof}
By K\"unneth Theorem we have the ring isomorphism
\red{\[\begin{split}
H^*(F)\bigotimes_{\mathbf{k}}\mathbf{k}\bigg[\Big(\text{Res}_{T_{i_1}}(\tau_{i_1})\Big)(p),\dots,\Big(\text{Res}_{T_{i_k}}(\tau_{i_k})\Big)(p)\bigg]&\to H^*\Big((BT_F\times\{p\})\times F\Big)
\end{split}\]}
Hence 
\begin{equation}\label{psi inverse tau ij F expression}
(\psi_{p}^*)^{-1}\Big(\tau_{i_j}(M_\mrF)\Big)=\sum_{l=1}^k\lambda_{jl}p_1^*\bigg(\Big(\text{Res}_{T_{i_l}}(\tau_{i_l})\Big)(p)\bigg)+p_2^*(\alpha_j)
\end{equation}
for some $\alpha_j\in H^2(F),$ where $\lambda_{jl}\in\mathbf{k}$, $p_1^* $ and $p_2^*$ are induced by projections. For proving this lemma, it suffices to show that 
\[\lambda_{jl}=\delta_{jl}.\]
For any $\alpha\in H^*(F)$ we have $\Big(\text{Res}_{T_F}(\psi_{p}^*p_2^*(\alpha))\Big)\Big|_{BT_F\times\{p\}}=0$.  Indeed, for sufficiently large $m$ we have $\alpha^m=0$ but $H^*(BT_F\times\{p\})$ has no nonzero nilpotent elements. Hence (\ref{psi inverse tau ij F expression}) implies that 
\begin{equation}\label{r_p action on both sides}
\text{Res}_{T_F}\Big(\tau_{i_j}(M_\mrF)\Big)=\sum_{l=1}^k\lambda_{jl}\cdot\text{Res}_{T_F}\Bigg(\psi_{p}^*p_1^*\bigg(\Big(\text{Res}_{T_{i_l}}(\tau_{i_l})\Big)(p)\bigg)\Bigg).
\end{equation}
Note that following diagram is commutative
\begin{center}
\begin{tikzcd}
 ET_F\times_{T_F}\{p\} \arrow[d, "\text{identity~map}"] \arrow[r, "{}"]  & ET_F\times_{T_F}M_{\mrF} \arrow[r, "{}"]  & ET\times_{T}M_{\mrF} \arrow[d, "\psi_{p}
 "]\\
 BT_F\times\{p\}  & & (BT_F\times\{p\}) \times F\arrow[ll, "p_1"]
\end{tikzcd}
\end{center}
where $ET_F\times_{T_F}\{p\}\to ET_F\times_{T_F}M_{\mrF}$ is given by the $T_F$-equivariant inclusion map $p\to M_{\mrF}$, $ET_F\times_{T_F}M_{\mrF}\to ET\times_{T}M_{\mrF}$ is determined by $ET_F\to ET$ and $(BT_F\times\{p\}) \times F\xrightarrow{p_1}  BT_F\times\{p\}$ is the first projection map, so we have the following commutative diagram.
\[\begin{tikzcd}
 H^*(ET_F\times_{T_F}\{p\})  &H^*(ET_F\times_{T_F}M_{\mrF}) \arrow[l, "{}"]    & H^*(ET\times_{T}M_{\mrF}) \arrow[l, "{}"]\\
 H^*(BT_F\times\{p\})\arrow[u, "\text{identity~map}"]\arrow[rr, "p_1^*"]  & & H^*((BT_F\times\{p\}) \times F)\arrow[u, "\psi_{p}^*
 "]
\end{tikzcd}\]
This show that for $\Big(\text{Res}_{T_{i_l}}(\tau_{i_l})\Big)(p)\in H^2(BT_F\times\{p\})$, we have 
\[\Bigg(\text{Res}_{T_F}\Bigg(\psi_{p}^*p_1^*\bigg(\Big(\text{Res}_{T_{i_l}}(\tau_{i_l})\Big)(p)\bigg)\Bigg)\Bigg)\Big|_{BT_F\times\{p\}}=\Big(\text{Res}_{T_{i_l}}(\tau_{i_l})\Big)(p).\]
Hence restriction both sides of (\ref{r_p action on both sides}) on $BT_F\times\{p\}$, we have
\[\bigg(\text{Res}_{T_F}\Big(\tau_{i_j}(M_\mrF)\Big)\bigg)\Big|_{BT_F\times\{p\}}=\sum_{l=1}^k\lambda_{jl}\Big(\text{Res}_{T_{i_l}}(\tau_{i_l})\Big)(p).\]
Since $\bigg(\text{Res}_{T_F}\Big(\tau_{i_j}(M_\mrF)\Big)\bigg)\Big|_{BT_F\times\{p\}}=\Big(\text{Res}_{T_{i_j}}(\tau_{i_j})\Big)(p)$ and $\Big(\text{Res}_{T_{i_1}}(\tau_{i_1})\Big)(p),\cdots, \Big(\text{Res}_{T_{i_k}}(\tau_{i_k})\Big)(p)$ is a basis of $H^2(BT_F\times\{p\})$, we obtain that 
\[\lambda_{jl}=\delta_{jl}.\]
This complete the proof of the lemma.
\end{proof}

\vspace{2mm}
\red{\begin{rema}
$\alpha_j$ in Lemma \ref{psi inverse tau ij F} is related to Yoshida's construction \cite{yo-11} for torus manifolds with locally standard action. 
\end{rema}}
\vspace{2mm}
\begin{proposition}\label{HF[x] to HTmrF}
The ring homomorphism
\[\begin{split}
\rho_F: H^*(F)[\mathbf{x}_F]&\to H^*_T(M_{\mrF})\\
\sum_{\mathbf{m}}\alpha(\mathbf{m})(\mathbf{x}_F)^{\mathbf{m}}&\mapsto\sum_{\mathbf{m}}\overline{q}^*(\alpha(\mathbf{m}))\Big(\tau(M_\mrF)\Big)^{\mathbf{m}},
\end{split}\]
is an isomorphism.
Here $\mathbf{m}=(m_1,\dots, m_k)\in\mathbb{Z}_{\ge 0}^k$, $\alpha(\mathbf{m})\in H^*(F)$, $(\mathbf{x}_F)^{\mathbf{m}}=x_{i_1}^{m_1}\dots x_{i_k}^{m_k}$,
$\Big(\tau(M_\mrF)\Big)^{\mathbf{m}}=\Big(\tau_{i_1}(M_\mrF)\Big)^{m_1}\cdots \Big(\tau_{i_k}(M_\mrF)\Big)^{m_k}$ and $\overline{q}^*$ is induced by the composition map $\overline{q}: ET\times_TM_\mrF\to \mrF \to F.$ 
\end{proposition}
\begin{proof}
Since $\psi_p: ET\times_TM_\mrF\to (BT_F\times\{p\})\times F$ is a homotopy equivalent, it is enough to show that $(\psi_p^*)^{-1}\rho_F$ is an isomorphism.
First we show that $(\psi_p^*)^{-1}\rho_F$ is an injection. If \[(\psi_p^*)^{-1}\rho_F\Big(\sum_{\mathbf{m}}\alpha(\mathbf{m})(\mathbf{x}_F)^{\mathbf{m}}\Big)
=0\]
then 
\[\sum_{\mathbf{m}}(\psi_{p}^*)^{-1}\bigg(\overline{q}^*(\alpha\Big(\mathbf{m})\Big)\bigg)\cup(\psi_{p}^*)^{-1}\bigg(\Big(\tau(M_\mrF)\Big)^{\mathbf{m}}\bigg)=0.\]
Therefore by (\ref{psi inverse widetilde qf}) and Lemma \ref{psi inverse tau ij F} , we have 
\[p_2^*\Big({\alpha(\mathbf{m})}\Big)\cup p_1^*\bigg(\Big(\text{Res}_{T_{i_1}}(\tau_{i_1})\Big)(p)+p_2^*(\alpha_1)\bigg)^{m_1}\cup\cdots\cup p_1^*\bigg(\Big(\text{Res}_{T_{i_k}}(\tau_{i_k})\Big)(p)+p_2^*(\alpha_k)\bigg)^{m_k}=0\]
where $\alpha_1,\dots,\alpha_k\in H^2(F)$ are determined by Lemma \ref{psi inverse tau ij F} .
This implies that for any $\mathbf{m}\in \Z_{\ge 0}^k$ we have 
\[\alpha(\mathbf{m})=0.\]
Hence 
\[\sum\limits_{\mathbf{m}}\alpha(\mathbf{m})(\mathbf{x}_F)^{\mathbf{m}}=0.\] This means that $\rho_F$ is injective. 

Next we show that $(\psi_p^*)^{-1}\rho_F$ is surjective. Since for any $\alpha\in H^2(F)$ we have\[
(\psi_{p}^*)^{-1}\rho_F(\alpha)=p_2^*(\alpha)
\] which follows from (\ref{psi inverse widetilde qf}).
Since
\[(\psi_{p}^*)^{-1}\rho_F(x_{i_j})=(\psi_{p}^*)^{-1}\Big(\tau_{i_j}(M_\mrF)\Big)=p_1^*\bigg(\Big(\text{Res}_{T_{i_j}}(\tau_{i_j})\Big)(p)\bigg)+p_2^*(\alpha_j)\]
where the second equality follows from Lemma \ref{psi inverse tau ij F}, we have 
\[(\psi_{p}^*)^{-1}\rho_F(x_{i_j}-
\alpha_j)=p_1^*\bigg(\Big(\text{Res}_{T_{i_j}}(\tau_{i_j})\Big)(p)\bigg).\]
Therefore the surjection of
\[(\psi_{p}^*)^{-1}\rho_F: H^*(F)[\mathbf{x}_F]\to H^*\Big((BT_F\times\{p\})\times F\Big)\] follows from K\"unneth Theorem

Since $(\psi_{p}^*)^{-1}\rho_F$ is both injective and surjective, it is an isomorphism. This implies that $\rho_F$ is an isomorphism.
\end{proof} 

\vspace{2mm}
The following is a direct corollary of the above proposition.
\begin{corollary}
Let $\mcF$ be the set of faces of $M/T$, then the ring homomorphism 
\[\bigoplus\limits_{F\in\mcF}\rho_F: \bigoplus\limits_{F\in\mcF} H^*(F)[\mathbf{x}_F]\to \bigoplus\limits_{F\in\mcF}H_T^*(M_{\mrF})
\]
is an isomorphism. In particular we can identify the topological face ring $\mathbf{k}[M/T]$ with the subring $(\bigoplus\limits_{F\in\mcF}\rho_F)\Big(\mathbf{k}[M/T]\Big)$ of $\bigoplus\limits_{F\in\mcF}H_T^*(M_{\mrF})$.
\end{corollary}

Hence we can express the topological face ring $\mathbf{k}[M/T]$ in terms of the orbit space $M/T$ and equivariant Thom classes.  
In the remain part of this section we translate   definition of topological face ring defined in Section \ref{topological face ring} in terms of $Q=M/T$ and equvariant Thom classes.

Let $\mcF$ be the set of faces of $M/T$ and 
$\mathring{\mcF}$ be the set of open faces of $M/T$. Let $\{F_1,\dots, F_m\}$ be the set of facets of $M/T$. Since $M/T$ is a nice manifold with corners, for each codimension $k$-face $F$ of $M/T$, where $k\ge 1$, $F$ is a connected component of $F_{i_1}\cap\dots\cap F_{i _k}$. Hence we can associated each face of $M/T$ to a subset of $[m].$ Namely we have a map
\[
\begin{split}
\Psi: \mcF&\to 2^{m}\\
F&\mapsto\{i_1,\dots, i_k\}.
\end{split}
\]
Let 
\[
B:=\bigoplus\limits_{F\in\mcF}H_T^*(M_{\mrF})
.\]
For each element $\alpha\in B$, we denote the $F$-component of $\alpha$ by $\alpha_F$.  Let $\alpha_F=\sum\limits_{\mathbf{m}}\overline{q}^*(a_{\mathbf{m},F})\Big(\tau(M_\mrF)\Big)^{\mathbf{m}}$, where $a_{\mathbf{m},F}\in H^*(F)$, $\mathbf{m}=(m_1,\dots, m_k)$, $\overline{q}^*$ is induced by the map $ET\times_TM_\mrF\to \mrF\to F$ and \red{$\Big(\tau(M_\mrF)\Big)^{\mathbf{m}}=\tau_{i_1}(M_\mrF)^{m_1}\dots \tau_{i_k}(M_\mrF)^{m_k}.$}
For an ordered pair of faces $(F,E)$, we  define a morphism $\phi_{FE}: H_T^*(M_\mrF)\to H_T^*(M_{\mathring{E}})$ as follows.
\begin{equation}\label{def of phiFE for new definition}
\phi_{FE}\left(\sum\limits_{\mathbf{m}}
\overline{q}^*(a_{\mathbf{m},F})\Big(\tau(M_\mrF)\Big)^{\mathbf{m}}\right)
=
\begin{cases}
\sum\limits_{\mathbf{m}}\overline{q}^*(a_{\mathbf{m},F}|_E)\Big(\tau(M_{\mathring{E}})\Big)^{\mathbf{m}}\quad\text{if}~E\subseteq F\\
0\hspace{46mm}\text{otherwise}
\end{cases}
\end{equation}
where $a_{\mathbf{m},F}|_E$ is the restriction of $a_{\mathbf{m},F}$ on $E$ induced by the inclusion map $E\subseteq F.$

\begin{definition}\label{def of F-face elements for new definition}
Let $F$ be a face of $M/T$ and $\Psi(F)=\{i_1,\dots, i_k\}$. An element $\alpha\in B$ is called an $F$-face element if the following two conditions hold:
\begin{itemize}
\item
$\alpha_F=\sum\limits_{\mathbf{m}}\overline{q}^*(a_{\mathbf{m},F})\Big(\tau(M_\mrF)\Big)^{\mathbf{m}}$ satisfying $\mathbf{m}\in\N_{>0}^k;$
\item
$\alpha_E=\phi_{FE}(\alpha_F).$
\end{itemize}
\end{definition}
\begin{definition}\label{topological face ring for new definition}
The subring of $B$ generated by the faces elements of $B$ is called a topological face ring, denoted by $\mathbf{k}[M/T].$ Namely
\[
\mathbf{k}[M/T]=\text{Span}_{\mathbf{k}}\langle\alpha: \alpha~\text{is a $F$-face element of $B$ for some $F\in\mcF$}\rangle.
\]
\end{definition}

\vspace{3mm}
\section{Relation between $H^*_T(M)$ and topological face rings}\label{sect: equivariant cohomology of M}
In this section we show our main result that the equivariant cohomology ring $H^*_T(M)$ is isomorphism to the topological face ring $\mathbf{k}[M/T]$, where $M/T$ is the orbit space of $M$. 
\subsection{The injective map $H^*_T(M)\to\bigoplus\limits_{F\in\mcF}H_T^*(M_{\mrF})$}
In general, for a torus manifold with locally standard action the restriction map $H^*_T(M)\to H^*_T(M^T)$
may not be injective. For instance, if $M$ is an orientable toric origami manifold such that each proper faces of $M/T$ is acyclic whereas $M/T$ is not acyclic, then $H^*_T(M)\to H^*_T(M^T)$ is not injective (see \cite{ay-ma-pa-ze-17}). However the restriction map
\[H^*_T(M)\to\bigoplus\limits_{F\in\mcF}H_T^*(M_{\mrF})\]
is an injective that we show in this subsection. We prove this injection by induction on $^\sharp\mcF$ the number of faces of $M/T$.
\begin{lemma}\label{sharp F(1)=sharp F-1}
Assume that $^\sharp\mcF\ge 2$ and let  $F_{\bullet}$ be a minimal face of $Q=M/T$. Then the following statements hold:
\begin{itemize}
\item
$F_\bullet=\mrF_\bullet;$
\item
$M(1):=M\setminus M_{\mrF_{\bullet}}$ is \red{an open} torus manifold with locally standard action;
\item
$^\sharp\mcF(1)\le^\sharp\mcF-1$, where $\mcF(1)$ the set of faces of $Q(1):=M(1)/T.$
\end{itemize}
\end{lemma}
\begin{proof}
Since $F_\bullet$ is a minimal face of $Q$, there is no face $E\in\mcF$ such that $E\subsetneq F_\bullet.$ we have $\mrF_\bullet=F_\bullet$. Hence the first statement of this lemma holds.

The first statement of this lemma show that $\mrF_\bullet$ is a closed subset of $Q$, which implies that  $M_{\mrF_\bullet}=M_{F_\bullet}$ is a closed $T$-invariant submanifold of $M$. Therefore $M(1)=M\setminus M_{\mrF_{\bullet}}$ is an open $T$-invariant submanifold of $M$, which \red{means that} $M(1)$ is a torus manifold with locally standard action. This proves the second claim of the lemma.

Note that $\mcF(1)=\{F\setminus F_{\bullet}: F\in\mcF\},$ so we have $^\sharp\mcF(1)\le{^\sharp}\mcF-1.$ This proves the third claim of the lemma. 
\end{proof}

\vspace{2mm}
Let $\mrF$ be an open face of $M/T$ and $M_\mrF=q^{-1}(\mrF)$ where $q: M\to M/T$ is the quotient map. Let $X_{\mrF}$ be a $T$-invariant tubular neighbourhood of $M_\mrF$ defined in (\ref{X mathring F}). 
\begin{lemma}\label{surjection of kappa F}
The restriction map
\[\kappa_{\mrF}: H^*_T(X_{\mrF})\to H^*_T(X_{\mrF}\setminus M_{\mrF})\]
is surjective.
\end{lemma}
\begin{proof}
Consider the long exact sequence for the pair $(X_\mrF, X_\mrF\setminus M_\mrF)$
\begin{equation}\label{long exact sequence for (XmrF, XmrF-MmrF)}
\begin{split}
\cdots&\to H^i_T(X_{\mrF}, X_{\mrF}\setminus M_{\mrF})\to H^i_T(X_{\mrF})\xrightarrow{\kappa_{\mrF}} H^i_T(X_{\mrF}\setminus M_{\mrF})\\
&\to H^{i+1}_T(X_{\mrF}, X_{\mrF}\setminus M_{\mrF})\to\cdots
\end{split}
\end{equation}
To show $\kappa_{\mrF}$ is surjective, it is enough to prove that 
\[H^i_T(X_{\mrF}, X_{\mrF}\setminus M_{\mrF})\to H^i_T(X_{\mrF})\]
is injective. Let $\Psi(F)=\{i_1,\dots,i_k\}$, that is, $F$ be a connected component of $F_{i_1}\cap\cdots\cap F_{i_k}$ where $F_{i_j}$'s are facets of $Q$. Consider the \red{following commutative diagram}.
\begin{center}
\begin{tikzcd}
 H^i_T(X_{\mrF}, X_{\mrF}\setminus M_{\mrF})   \arrow[r, "{}"]  & H^i_T(X_{\mrF}) 
 \\
H^{i-2k}_T(M_{\mrF}) \arrow[u, "\text{}"]   \arrow[r, "{\cup e^T}"]&  H^i_T(M_{\mrF})\arrow[u, "\pi_{\mrF}^*"]
\end{tikzcd}
\end{center}
where $H^{i-2k}_T(M_{\mrF})\to H^i_T(X_{\mrF}, X_{\mrF}\setminus M_{\mrF})$ is the \red{equivariant Thom isomorphism}, $e^T$ is the $T$-equivariant Euler class of the normal bundle $\nu_\mrF$ of $M_{\mrF}$ in $M$, whose total space can identify with $X_{\mrF}$ and 
$\pi^*_\mrF$ is induced by the projection $\pi_\mrF: X_\mrF\to M_\mrF.$

  Since $\pi_{\mrF}^*$ is an isomorphism, to show that $H^i_T(X_{\mrF}, X_{\mrF}\setminus M_{\mrF})\to H^i_T(X_{\mrF})$ is injective it suffices to show that 
\[H^{i-2k}_T(M_{\mrF})\xrightarrow{\cup e^T}H^i_T(M_{\mrF})\]
is injective. By the decomposition (\ref{vector bundle decomposition}) we have
$\nu_{\mrF}\cong\bigoplus_{j=1}^k(\nu_{F_{i_j}}|_{M_{\mrF}}),$
which implies that 
\red{$e^T=\prod_{j=1}^k\tau_{i_j}(M_\mrF)$
where $\tau_{i_j}(M_\mrF)$ is defined \red{at the beginning of subsection \ref{subsection: On HTF(p)}.}}
Hence the injection of the map $H^{i-2k}_T(M_{\mrF})\xrightarrow{\cup e^T}H^i_T(M_{\mrF})$ follows from Proposition \ref{HF[x] to HTmrF}. This completes the proof of the lemma.
\end{proof}

\vspace{2mm}
\red{Let $F$ be a codimension $k$} face of $M/T$ and $\tau_F(X_{\mrF})$ be the image of $1$ given by the composition map
\red{\begin{equation}\label{eq: definition of tauFXmrF by gysin map}
H^0_T(M_F)\to H_T^{2k}(M)\to H_T^{2k}(X_{\mrF}),
\end{equation}}
where the first homomorphism is the equivariant Gysin map and the second homomorphism is induced \red{by} the inclusion map $X_{\mrF}\subseteq M.$
\red{\begin{lemma}\label{kernel of kappa F generated by Thom class}
The kernel of $\kappa_{\mrF}: H^*_T(X_{\mrF})\to H^*_T(X_{\mrF}\setminus M_{\mrF})$ is $\Big(\tau_{F}(X_{\mrF})\Big),$
where $\Big(\tau_{F}(X_{\mrF})\Big)$ is the ideal, of $H^*_T(X_{\mrF})$ generated by $\tau_{F}(X_{\mrF})$.
\end{lemma}}
\begin{proof}
By the long exact sequence (\ref{long exact sequence for (XmrF, XmrF-MmrF)}), it suffices to show that the image of 
$H^*_T(X_{\mrF}, X_{\mrF}\setminus M_{\mrF})\to H^*_T(X_{\mrF})$ is $\Big(\tau_{F}(X_{\mrF})\Big)$. By the arguments of the proof of Lemma \ref{surjection of kappa F}, the image of $H^*_T(X_{\mrF}, X_{\mrF}\setminus M_{\mrF})\to H^*_T(X_{\mrF})$ is exact $\pi_{\mrF}^*\Big((e^T_{\mrF})\Big)$, where $(e^T_{\mrF})$ is the $T$-equivariant Euler class of the normal bundle $\nu_\mrF$ of $M_{\mrF}$ in $M$ and $(e^T_{\mrF})$ is an ideal, generated by $(e^T_{\mrF})$, of $H^*_T(M_\mrF).$ Therefore it is enough to show that 
\[\pi_{\mrF}^*(e^T_{\mrF})=\tau_{F}(X_{\mrF}).\]
Consider the following commutative diagram
\begin{center}
	\begin{tikzcd}
		H^{2k}_T(M, M\setminus M_{F}) \arrow[rrr] \arrow[d, "\cong"']&&& H^{2k}_T(M) \arrow[d]\\
		H^{2k}_T(X_{F}, X_{F}\setminus M_{F}) \arrow[dr] \arrow[rrr] &  & & H^{2k}_T(X_{F}) \arrow[dl]\\
		& H^{2k}_T(X_{\mrF}, X_{\mrF}\setminus M_{\mrF}) \arrow[r] & H^{2k}_T(X_{\mrF}) & \\
		& H^{0}_T(M_{\mrF}) \arrow[u,"{\cong}"] \arrow[r,"{\cup e^T_{\mrF}}"] & H^{2k}_T(M_{\mrF}) \arrow[u,"\pi_{\mrF}^*"] & \\
		H^{0}_T(M_{F})  \arrow[uuu,"{\cong}"] \arrow[ur]\arrow[rrr,"{\cup e^T_F}"]& & & H^{2k}_T(M_{F}) \arrow[uuu,"\pi_{F}^*"] \arrow[ul],
	\end{tikzcd}
\end{center}
where $(e^T_F)$ is the $T$-equivariant Euler class of the normal bundle $\nu_F$ of $M_{F}$ in $M$, $2k$ is the rank of $\nu_F$ as a real vector bundle, the isomorphisms in the commutative diagram are from excision and Thom isomorphism. Hence 
$\pi_{\mrF}^*(e^T_{\mrF})$ is the image of $1\in H^0_T(M_F)$ by the composition map
\[H^0_T(M_F)\to H^{2k}_T(M_F)\to H^{2k}_T(M_\mrF)\to H^{2k}_T(X_\mrF).\]
\red{By (\ref{eq: definition of tauFXmrF by gysin map}) and the above commutative diagram,} $\tau_{F}(X_{\mrF})$ is 
the image of $1\in H^0_T(M_F)$ by the composition map 
\[H^0_T(M_F)\to H^{2k}_T(X_F,X_F\setminus M_F)\to H^{2k}_T(M,M\setminus M_F)\to H^{2k}_T(M)\to H^{2k}_T(X_F)\to H^{2k}_T(X_\mrF).\] 
Therefore \[\pi_{\mrF}^*(e^T_{\mrF})=\tau_{F}(X_{\mrF}).\]
\end{proof}
\vspace{2mm}
\begin{proposition}\label{injective observation}
Let $M$ ba an open torus manifold with locally standard action. Then the restriction map
\[
\bigoplus\limits_{F\in\mcF} r_{\mrF}: H^*_T(M)\to\bigoplus\limits_{F\in\mcF}H_T^*(M_{\mrF})
\]
is injective, where 
$\mcF$ is the set of faces of $Q$ and $r_\mrF: H^*_T(M)\to H^*_T(M_\mrF)$ is the restriction map induced by the inclusion map $M_\mrF\subseteq M.$
\end{proposition}
\begin{proof}
We prove this proposition by induction on $^\sharp\mcF$, the number of elements of $\mcF$. 

When $^\sharp\mcF=1$, there is no proper face of $Q=M/T$, so $M=M_{\mathring{Q}}$. Hence $H_T^*(M)=H^*_T(M_{\mathring{Q}})$ which implies that the restriction map
\[r_{\mathring{Q}}: H_T^*(M)\to H^*_T(M_{\mathring{Q}})\]
is an identity map. In particular, $r_{\mathring{Q}}$ is injective.

Assume that $^\sharp\mcF\ge 2$. Let $F_{\bullet}$ be a minimal face of $Q=M/T$ and $M(1):=M\setminus M_{\mrF_{\bullet}}$. Then by Lemma \ref{sharp F(1)=sharp F-1}, we have $^\sharp\mcF(1)\le^\sharp\mcF-1$, where 
$\mcF(1)~\text{is the set of faces of $Q(1):=M(1)/T.$}$
Hence  by induction hypothesis we have the injective map
\begin{equation}\label{injection of oplus rmrF}
\bigoplus\limits_{F\in\mcF(1)} r_{\mrF}: H^*_T(M(1))\to\bigoplus\limits_{F\in\mcF(1)} H^*_T(M_{\mrF}).
\end{equation}
Since $M=M(1)\cup X_{\mrF_\bullet}$ and $M(1)\cap X_{\mrF_\bullet}=X_{\mrF_\bullet}\setminus M_{\mrF_\bullet}$, we have the  Mayer–Vietoris exact sequence 
\begin{equation}\label{M-V sequence}
\begin{split}
\cdots\to& H^i_T(M)\to H^i_T(M(1))\oplus H^i_T(X_{\mrF_\bullet})\to H^i_T(X_{\mrF_\bullet}\setminus M_{\mrF_\bullet})\\
\to&H^{i+1}_T(M)\to\cdots
\end{split}.
\end{equation} 
By Lemma \ref{surjection of kappa F} for any $i\ge 0$
\[H^i_T(M(1))\oplus H^i_T(X_{\mrF_\bullet})\to H^i_T(X_{\mrF_\bullet}\setminus M_{\mrF_\bullet})\]
is surjective, so 
\[H^i_T(M)\to H^i_T(M(1))\oplus H^i_T(X_{\mrF_\bullet})\]
is injective.
Note that the following diagram is commutative

\begin{center}
\begin{equation}\label{commutative diagram for HT*(M) and HT*(M(1)) and HT*(XF)}
\begin{tikzcd}
 H^i_T(M)   \arrow[d, "{\bigoplus\limits_{F\in\mcF} r_{\mrF}}"]\arrow[r, "{}"]  & H^i_T(M(1))\bigoplus H^i_T(X_{\mrF_\bullet})\arrow[d, "{(\bigoplus\limits_{F\in\mcF(1)} r_{\mrF})\bigoplus(\pi_{\mrF_\bullet}^*)^{-1}}"]\\
\bigoplus\limits_{F\in\mcF}H^{i}_T(M_{\mrF})    \arrow[r, "{=}"]& \Big(\bigoplus\limits_{F\in\mcF(1)}H^i_T(M(1)_{\mrF})\Big)\bigoplus H^i_T(M_{\mrF_\bullet}).\end{tikzcd}
\end{equation}
\end{center}
Since the maps (\ref{injection of oplus rmrF}) is injective by induction hypothesis, $H^i_T(M)\to H^i_T(M(1))\oplus H^i_T(X_{\mrF_\bullet})$ is injective which we have shown and $(\pi_{\mrF}^*)^{-1}$ is an isomorphism, we obtain that 
 \[\bigoplus\limits_{F\in\mcF} r_{\mrF}:  H^i_T(M)\to \bigoplus\limits_{F\in\mcF}H^{i}_T(M_{\mrF})\] is injective by the above commutative diagram. This completes the proof of Proposition \ref{injective observation}.
\end{proof}

\vspace{3mm}
\subsection{The image of the restriction map $H^*_T(M)\to\bigoplus\limits_{F\in\mcF}H_T^*(M_{\mrF})$} 
In the last subsection, we show that the restriction map $\bigoplus\limits_{F\in\mcF} r_{\mrF}: H^*_T(M)\to\bigoplus\limits_{F\in\mcF}H_T^*(M_{\mrF})$ is injective, so $H^*_T(M)$ is isomorphic to $\text{Im}(\bigoplus\limits_{F\in\mcF} r_{\mrF}).$ In this subsection we show that $\text{Im}(\bigoplus\limits_{F\in\mcF} r_{\mrF})$ is isomorphic to the topological face ring $\mathbf{k}[M/T].$ We divide the proof into two parts that are $\mathbf{k}[M/T]\subseteq \text{Im}(\bigoplus\limits_{F\in\mcF} r_{\mrF})$ and $\text{Im}(\bigoplus\limits_{F\in\mcF} r_{\mrF})\subseteq \mathbf{k}[M/T]$.

\vspace{2mm}

\red{Let $F_{\bullet}$ be a minimal face of $Q=M/T$, $M_\mrF=q^{-1}(\mrF)$ and $X_{\mrF}=\pi_{F}^{-1}(M_{\mrF})$ defined in \ref{X mathring F}}
\begin{lemma}\label{injection of r}
The map restriction map 
\[r: H^i_T(X_{\mrF_\bullet}\setminus M_{\mrF_\bullet})\to \bigoplus\limits_{F\in\mcF(1)}H^i_T(M(1)_{\mrF}\cap X_{\mrF_{\bullet}})\]
is injective.
\end{lemma}
\begin{proof}
We claim that if $F_\bullet\nsubseteq F$ then $M_{\mrF}\cap X_{\mrF_\bullet}=\emptyset.$ Indeed, if $p\in X_{\mrF_\bullet}$, then by Lemma \ref{X mathring F T-equivariant diff to nu mathring F} and Corollary \ref{local model for faces} the isotropy subgroup of $p$ is a subgroup of $T_{F_\bullet}$. However for $q\in M_{\mrF}$, the isotropy subgroup of $q$ is $T_F$ which is not a subgroup of $T_{F_\bullet}$. This implies that $M_{\mrF}\cap X_{\mrF_\bullet}=\emptyset$, so the claim holds.

Hence we have the following $T$-orbit decomposition
\[X_{\mrF_\bullet}\setminus M_{\mrF_\bullet}=\bigsqcup_{F\in\mcF(1), F_\bullet\subsetneq F}(M(1)_{\mrF}\cap X_{\mrF_\bullet}).\]
Since $M_F\cap X_{F_\bullet}$ 
can identify with the normal bundle of $M_{F_\bullet}$ in $M_F$, $M_F\cap X_{F_\bullet}$ is connected. Note that 
$M(1)_{\mrF}\cap X_{F_\bullet}$ is obtained from $M_F\cap X_{F_\bullet}$ by removing some submanifolds with codimension at least 3, so $M(1)_{\mrF}\cap X_{F_\bullet}$ is also connected. Therefore the set of open faces of
$X_{\mrF_\bullet}\setminus M_{\mrF_\bullet}$ is exact 
\[\{M(1)_{\mrF}\cap X_{\mrF_\bullet})/T: F\in\mcF(1), F_\bullet\subsetneq F\}.\]
Therefore this lemma follows from Proposition \ref{injective observation}.
\end{proof}

\vspace{2mm}
By the previous arguments, we see that the Mayer-Vietoris exact sequence (\ref{M-V sequence}) can split into the short exact sequence 
\begin{equation}\label{split into short exact sequence}
0\to H_T^*(M)\to H^*_T(M(1))\oplus H^*_T(X_{\mrF_\bullet})\xrightarrow{r_1-\kappa_{\mrF_\bullet}} H^*_T(X_{\mrF_\bullet}\setminus M_{\mrF_\bullet})\to0.
\end{equation}
where $r_1: H^*_T(M(1))\to H^*_T(X_{\mrF_\bullet}\setminus M_{\mrF_\bullet})$ and $\kappa_{\mrF_\bullet}:H^*_T(X_{\mrF_\bullet})\to H^*_T(X_{\mrF_\bullet}\setminus M_{\mrF_\bullet})$ are restriction maps induced by inclusions $X_{\mrF_\bullet}\setminus M_{\mrF_\bullet}\subseteq M(1)$ and $X_{\mrF_\bullet}\setminus M_{\mrF_\bullet}\subseteq X_{\mrF_\bullet}$ respectively. Moreover, we have the following commutative diagram that will be used later,
\begin{center}
\begin{equation}\label{commutative diagram for r(1)xF-kappa F}
\begin{tikzcd}
 H^i_T(M(1))\oplus H^i_T(X_{\mrF_\bullet})  \arrow[d, "{(\bigoplus\limits_{F\in\mcF(1)} r_{\mrF})\bigoplus id}"']\arrow[rrr, "{r_1-\kappa_{\mrF_\bullet}}"]  &&& H^i_T(X_{\mrF_\bullet}\setminus M_{\mrF_\bullet})\arrow[d, "{r}"]\\
\Big(\bigoplus\limits_{F\in\mcF(1)}H^i_T(M(1)_{\mrF})\Big)\bigoplus H^i_T(X_{\mrF_\bullet}) \arrow[rrr, "{\bigoplus\limits_{F\in\mcF(1)}(h^1_{\mrF}-h^2_{\mrF})}"]&&& \bigoplus\limits_{F\in\mcF(1)}H^i_T(M(1)_{\mrF}\cap X_{\mrF_{\bullet}})\end{tikzcd}
\end{equation}
\end{center}
where all the maps are induced by the inclusion maps,
$h^1_\mrF: H^*_T(M(1)_{\mrF}))\to H^*_T(M(1)_{\mrF}\cap X_{\mrF_\bullet})
$
and $h^2_{\mrF}:H^*_T(X_{\mrF_\bullet})\to H^*_T(M(1)_{\mrF}\cap X_{\mrF_\bullet})$. 
\begin{corollary}\label{H*T(M) as ker cap Im}
Let the $\mathcal{R}$ be the kernel of the map 
\[
r_1-\kappa_{\mrF_\bullet}:
H^*_T(M(1))\oplus H^*_T(X_{\mrF_\bullet})\to H^*_T(X_{\mrF_\bullet}\setminus M_{\mrF_\bullet})
\]
Then
\[
\Big((\bigoplus\limits_{F\in\mcF(1)} r_{\mrF})\oplus id\Big)(\mathcal{R})=\ker\Big(\bigoplus\limits_{F\in\mcF(1)}(h^1_{\mrF}-h^2_{\mrF})\Big)\cap \text{Im}\Big((\bigoplus\limits_{F\in\mcF(1)} r_{\mrF})\bigoplus id\Big).\]
\end{corollary}
\begin{proof}
Note that two vertical maps in the above commutative diagram are injective, so we can obtain the following corollary which follows from direct verification.
\end{proof}
\vspace{2mm}
Let $\alpha\in\bigoplus\limits_{F\in\mcF}H^i_T(M_{\mrF})$ and 
$\widehat{\alpha}$ be the image of $\alpha$ through the isomorphism
\[\Big(\bigoplus\limits_{F\in\mcF(1)}H^i_T(M(1)_{\mrF})\Big)\bigoplus H^i_T(M_{\mrF_\bullet})\to \Big(\bigoplus\limits_{F\in\mcF(1)}H^i_T(M(1)_{\mrF})\Big)\bigoplus H^i_T(X_{\mrF_\bullet}).\]
\begin{lemma}\label{E-face element in kernel}
Let $E$ be a face of $M/T$ and $\alpha$ is an $E$-face element. Then 
\[
\widehat{\alpha}\in\ker\Big(\bigoplus\limits_{F\in\mcF(1)}(h^1_{\mrF}-h^2_{\mrF})\Big).
\]
\end{lemma}
\begin{proof}
Let $\beta=\Big(\bigoplus\limits_{F\in\mcF(1)}(h^1_{\mrF}-h^2_{\mrF})\Big)(\widehat{\alpha})$. Then for any $F\in\mcF(1)$, we have 
\[
\beta_F=h^1_{\mrF}(\widehat{\alpha}_F)-h^2_{\mrF}(\widehat{\alpha}_{F_\bullet})=h^1_{\mrF}(\alpha_F)-h^2_{\mrF}\Big(\pi_{\mrF_\bullet}^*(\alpha_{F_\bullet})\Big)\]
where $\alpha_F,\beta_F$ and $\widehat{\alpha}_F$ denote the $F$-component of $\alpha$, \red{$\beta$} and $\widehat{\alpha}$ respectively.

{\bf{Case 1}:} $F_\bullet\nsubseteq E$. Then by the definition of face element (see Definition \ref{def of F-face elements for new definition}) we have $\alpha_{F_\bullet}=0,$ which means that 
\[
h^2_{\mrF}(\widehat{\alpha}_{F_\bullet})=h^2_{\mrF}\Big(\pi^*_\mrF(\alpha_{F_\bullet})\Big)=0.\]
Since $F_\bullet$ is a minimal face of $Q$, we have $F_\bullet\cap E=\emptyset.$ If $F$ is a face of $E$, then $F\cap F_\bullet=\emptyset$ which implies that $h^1_{\mrF}=0$. Hence $h^1_{\mrF}(\alpha_F)=0$. If $F$ is not a face of $E$, then by the definition of face element, we have $\alpha_F=0$. Also we obtain that 
\[h^1_{\mrF}(\alpha_F)=0.\]
Therefore for this case we have 
\[
\beta_F=h^1_{\mrF}(\alpha_F)-h^2_{\mrF}(\alpha_{F_\bullet})=0.
\]

{\bf{Case 2}:} $F_\bullet\subseteq E$. 

{\bf{Case 2-1}:}
$F$ is not a face of $E$. Then we have $\alpha_F=0$
by the definition of $E$-face elements. Hence we have
\[
h^1_{\mrF}(\alpha_F)=0.
\]
Since $\alpha$ is an $E$-face element, we have
$\alpha_E=\sum\limits_{\mathbf{m}}\overline{q}^*(a_{\mathbf{m},E})\Big(\tau(M_{\mathring{E}})\Big)^{\mathbf{m}}$ where $\mathbf{m}\in\N_{>0}^k$ and \red{$a_{\mathbf{m},E}\in H^*(E)$}.   Hence $\alpha_{F_\bullet}=\sum\limits_{\mathbf{m}}\overline{q}^*(a_{\mathbf{m},E}|_{F_\bullet})\Big(\tau(M_{\mrF_\bullet})\Big)^{\mathbf{m}}$ and $\widehat{\alpha}_{F_\bullet}=\sum\limits_{\mathbf{m}}\overline{q}^*(a_{\mathbf{m},E}|_{F_\bullet})\Big(\tau(X_{\mrF_\bullet})\Big)^{\mathbf{m}}$.
Since $F$ is not a face of $E$, there exists $i\in\Psi(E)\setminus\Psi(F)\subseteq\Psi(F_\bullet)$. This shows that $i\notin\Psi(F)\cap\Psi(F_\bullet)$. Then by the definition of $h^2_{\mrF}$, we have
$h^2_{\mrF}\Big(\tau_i(X_\mrF)\Big)=0$. Therefore for any $\mathbf{m}\in\N_{>0}^k$, we have 
$
h^2_{\mrF}\bigg(\Big(\tau(X_\mrF)\Big)^{\mathbf{m}}\bigg)=0
$ which implies that 
\[
h^2_{\mrF}(\widehat{\alpha}_{F_\bullet})=0.
\]
In this case we have 
\[
\beta_F=h^1_{\mrF}(\alpha_F)-h^2_{\mrF}(\widehat{\alpha}_{F_\bullet})=0.\]

{\bf{Case 2-2}:} $F$ is a face of $E$ but $F_\bullet\nsubseteq F$. Since $F_\bullet$ is a minimal face, $F_\bullet\cap F=\emptyset,$ which means that $h^1_{\mrF}=0$ and $h^2_{\mrF}=0$. Therefore for this case we have 
\[
\beta_F=h^1_{\mrF}(\alpha_F)-h^2_{\mrF}(\widehat{\alpha}_{F_\bullet})=0.\]

{\bf{Case 2-3}:} $F$ is a face of $E$ and $F_\bullet\subseteq F$. Set $\alpha_E=\sum\limits_{\mathbf{m}}\overline{q}^*(a_{\mathbf{m},E})\Big(\tau(M_{\mathring{E}})\Big)^{\mathbf{m}}$ satisfying $\mathbf{m}\in\N_{>0}^k$. Then  
$\alpha_F=\sum\limits_{\mathbf{m}}\overline{q}^*(a_{\mathbf{m},E}|_F)\Big(\tau(M_{\mathring{F}})\Big)^{\mathbf{m}},$ where $\overline{q}^*$
is induced by the composition map 
\[ET\times_T M_\mrF\to\mrF\to F.\]
 Hence
\[
h^1_{\mrF}(\alpha_F)=\sum\limits_{\mathbf{m}}\overline{q}^*\Big(a_{\mathbf{m},E}\Big|_{(M_F\cap X_{F_\bullet})/T}\Big)\Big(\tau(M_{\mathring{F}})\Big|_{M(1)_{\mrF}\cap X_{\mrF_{\bullet}}}\Big)^{\mathbf{m}},\] 
where $\overline{q}^*$, by abusing notations, is induced by the map 
\[ET\times_T(M_\mrF\cap X_{\mrF_\bullet})\to (M_\mrF\cap X_{\mrF_\bullet})/T\to (M_F\cap X_{F_\bullet})/T.\]

By the same argument we have
$
\alpha_{F_\bullet}=\sum\limits_{\mathbf{m}}\overline{q}^*(a_{\mathbf{m},E}|_{F_\bullet})\Big(\tau(M_{\mrF_\bullet})\Big)^{\mathbf{m}}
$ and
\[
h^2_{\mrF}(\widehat{\alpha}_{F_\bullet})=\sum\limits_{\mathbf{m}}\overline{q}^*\Big(\mathbf{P}^*(a_{\mathbf{m},E}|_{F_\bullet})\Big)\left(\tau(X_{\mrF_\bullet})\Big|_{M_\mrF\cap X_{\mrF_\bullet}}\right)^{\mathbf{m}}
.\]
Here $\mathbf{P}^*$ is induced by the composition map  
\[\mathbf{P}: (M(1)_F\cap X_{F_\bullet})/T\to X_{F_\bullet}/T\xrightarrow{\overline{\pi_{F_\bullet}}} M_{F_\bullet}/T=F_\bullet\] where $\overline{\pi_{F_\bullet}}: X_{F_\bullet}/T\to M_{F_\bullet}/T=F_\bullet$ induced by $\pi_{F_\bullet}: X_{F_\bullet}\to M_{F_\bullet}$
and $\overline{q}^*$, by abusing notations, is induced by the map 
\[ET\times_T(M_\mrF\cap X_{\mrF_\bullet})\to (M_\mrF\cap X_{\mrF_\bullet})/T\to (M_F\cap X_{F_\bullet})/T.\]
Since $(M_F\cap X_{F_\bullet})/T$ is a face of nice manifolds with corner $X_{F_\bullet}/T$. By Corollary \ref{local model for faces},  the map 
\[(M_F\cap X_{F_\bullet})/T\xrightarrow{\mathbf{P}} F_\bullet\to (M_F\cap X_{F_\bullet})/T,\]
where $F_\bullet\to (M_F\cap X_{F_\bullet})/T$ is the inclusion map, is homotopy equivalent to the identity map 
\[(M_F\cap X_{F_\bullet})/T\to (M_F\cap X_{F_\bullet})/T.\]
This implies that  $\mathbf{P}^*(a_{\mathbf{m},E}|_{F_\bullet})=a_{\mathbf{m},E}\Big|_{(M_F\cap X_{F_\bullet})/T}.$ Therefore we have 
\[
h^2_{\mrF}(\widehat{\alpha}_{F_\bullet})=\sum\limits_{\mathbf{m}}\overline{q}^*\Big(a_{\mathbf{m},E}\Big|_{(M_F\cap X_{F_\bullet})/T}\Big)\left(\tau(X_{\mrF_\bullet})\Big|_{M_\mrF\cap X_{\mrF_\bullet}}\right)^{\mathbf{m}}
.\]
Then we have 
\[
\beta_F=h^1_{\mrF}(\alpha_F)-h^2_{\mrF}(\widehat{\alpha}_{F_\bullet})=0\]
for this case.

The above arguments show that
$\beta=0$ which implies this lemma.
\end{proof}

\vspace{2mm}
\begin{lemma}\label{one key lemma}
Let $\bigoplus\limits_{F\in\mcF} r_{\mrF}: H^*_T(M)\to\bigoplus\limits_{F\in\mcF}H_T^*(M_{\mrF})$ be the restriction map induced by inclusion maps. Then
\[\mathbf{k}[M/T]\subseteq \text{Im}(\bigoplus\limits_{F\in\mcF} r_{\mrF})
.\]
\end{lemma}
\begin{proof}
We prove this lemma by induction on $^\sharp\mcF$, where $\mcF$ is the set of faces of $Q=M/T$. When $^\sharp\mcF=1$, $M=\mathring{M}.$ Hence we have 
\[
H^*_T(M)=H^*_T(M_{\mathring{Q}})=\mathbf{k}[Q].\]
Therefore this lemma holds for $^\sharp\mcF=1$.

Assume that $^\sharp\mcF\ge 2$. Let $F_{\bullet}$ be a minimal face of $Q=M/T$ and $M(1):=M\setminus M_{\mrF_{\bullet}}$ Then by Lemma \ref{sharp F(1)=sharp F-1}, we have $^\sharp\mcF(1)\le^\sharp\mcF-1$, where 
$
\mcF(1)~\text{is the set of faces of $Q(1):=M(1)/T.$}$ Hence  by induction hypothesis we have 
\[
\mathbf{k}[Q(1)]\subseteq\text{Im}(\bigoplus\limits_{F\in\mcF(1)} r_{\mrF}).
\]
Let $E$ be a face of $Q$ and $\alpha$ be $E$-face element in $\bigoplus\limits_{F\in\mcF} H^*_T(M_
\mrF)$ so it suffices to show that $\alpha\in\text{Im}(\bigoplus\limits_{F\in\mcF} r_{\mrF})$ 

{\bf{Case 1:}} $F_\bullet\nsubseteq E.$ In this case $\alpha_{F_\bullet}=0.$ Hence we can view $\alpha$ as an element in of $\mathbf{k}[Q(1)].$ 
By induction hypothesis we have 
\[\alpha\in\text{Im}(\bigoplus\limits_{F\in\mcF(1)} r_{\mrF})\subseteq \text{Im}(\bigoplus\limits_{F\in\mcF} r_{\mrF}),
\]
where the last inclusion follows from the fact that $M(1)_\mrF=M_\mrF$ for $F\in\mcF(1).$

{\bf{Case 2:}} $F_\bullet\subseteq E.$ 
Let $\widehat{\alpha}$ be the image of $\alpha$ through the isomorphism
\[\Big(\bigoplus\limits_{F\in\mcF(1)}H^i_T(M(1)_{\mrF})\Big)\bigoplus H^i_T(M_{\mrF_\bullet})\to \Big(\bigoplus\limits_{F\in\mcF(1)}H^i_T(M(1)_{\mrF})\Big)\bigoplus H^i_T(X_{\mrF_\bullet})\] 
and $\alpha(1)$ be the element in $\mathbf{k}[Q(1)]$ by removing the $F_\bullet$-component of $\alpha$, so we express $\widehat{\alpha}$ as $\Big(\alpha(1),\pi^*_{\mrF_\bullet}(\alpha_{F_\bullet})\Big).$ By induction hypothesis, there is an element $\gamma\in H^*_T(M(1))$, such that 
\begin{equation}\label{exists gamma in M(1) such that}
\Big(\bigoplus\limits_{F\in\mcF(1)} r_{\mrF}\Big)(\gamma)=\alpha(1).
\end{equation}
Since $\Big(\bigoplus\limits_{F\in\mcF(1)}(h^1_{\mrF}-h^2_{\mrF})\Big)\Big(\alpha(1),\pi^*_{\mrF_\bullet}(\alpha_{F_\bullet})\Big)
=0
$
by Lemma \ref{E-face element in kernel}, we have 
\[r\circ(r_1-\kappa_{\mrF_\bullet})\Big(\gamma,\pi^*_{\mrF_\bullet}(\alpha_{F_\bullet})\Big)=0\]
which follows from commutative diagram (\ref{commutative diagram for r(1)xF-kappa F}). $r$ is injective by Lemma \ref{injection of r}, so 
\[(r_1-\kappa_{\mrF_\bullet})\Big(\gamma,\pi^*_{\mrF_\bullet}(\alpha_{F_\bullet})\Big)=0.\]
By short exact sequence \ref{split into short exact sequence} we have 
\[\Big(\gamma,\pi^*_{\mrF_\bullet}(\alpha_{F_\bullet})\Big)\in\text{Im}\Big(H^*_T(M)\to H^*_T(M(1))\oplus H^*_T(X_{\mrF_\bullet})\Big).\]
Applying commutative diagram
(\ref{commutative diagram for HT*(M) and HT*(M(1)) and HT*(XF)}), we have
\[\left(\Big(\bigoplus\limits_{F\in\mcF(1)} r_{\mrF}\Big)(\gamma),(\pi^*_{\mrF_\bullet})^{-1}\Big(\pi^*_{\mrF_\bullet}(\alpha_{F_\bullet})\Big)\right)\in\text{Im}(\bigoplus\limits_{F\in\mcF}r_{\mrF}).
\]
By (\ref{exists gamma in M(1) such that}), we have $(\alpha(1),\alpha_{F_\bullet})\in \text{Im}(\bigoplus\limits_{F\in\mcF}r_{\mrF}).$ Namely $\alpha\in\text{Im}(\bigoplus\limits_{F\in\mcF}r_{\mrF}).$

Summarizing the above arguments, for any face element $\alpha$ we have  
\[\alpha\in\text{Im}(\bigoplus_{F\in\mcF}r_{\mrF}),\]
so lemma follows from the fact that $\mathbf{k}[M/T]$ is generated by face elements.
\end{proof} 

\vspace{2mm}
\begin{lemma}\label{another key lemma}
Let $\bigoplus\limits_{F\in\mcF} r_{\mrF}: H^*_T(M)\to\bigoplus\limits_{F\in\mcF}H_T^*(M_{\mrF})$ be the restriction map induced by inclusion maps. Then
\[
\text{Im}(\bigoplus\limits_{F\in\mcF} r_{\mrF})\subseteq \mathbf{k}[M/T].
\]
\end{lemma}
\begin{proof}
We prove this lemma by induction on $^\sharp\mcF$, where $\mcF$ is the set of faces of $Q=M/T$. When $^\sharp\mcF=1$, then $M=\mathring{M}.$ Hence we have 
\[
H^*_T(M)=H^*_T(M_{\mathring{Q}})=\mathbf{k}[Q]. \]
Therefore this lemma holds for $^\sharp\mcF=1$.

Assume that $^\sharp\mcF\ge 2$. Let $F_{\bullet}$ be a minimal face of $Q=M/T$ and $M(1):=M\setminus M_{\mrF_{\bullet}}.$ Then by Lemma \ref{sharp F(1)=sharp F-1}, we have $^\sharp\mcF(1)\le {^\sharp\mcF}-1$, where $
\mcF(1)~\text{is the set of faces of $Q(1):=M(1)/T.$}$ Hence  by induction hypothesis we have 
\[
\text{Im}(\bigoplus\limits_{F\in\mcF(1)} r_{\mrF})\subseteq\mathbf{k}[Q(1)].
\]
This inclusion implies that
\[
\text{Im}\Big((\bigoplus\limits_{F\in\mcF(1)} r_{\mrF})\bigoplus id\Big)\subseteq \mathbf{k}[Q(1)]\bigoplus H^*_T(X_{\mrF_\bullet})
,\]
where  
\[(\bigoplus\limits_{F\in\mcF(1)} r_{\mrF})\bigoplus id: H^i_T(M(1))\oplus H^i_T(X_{\mrF_\bullet})\to \Big(\bigoplus\limits_{F\in\mcF(1)}H^i_T(M(1)_{\mrF})\Big)\bigoplus H^i_T(X_{\mrF_\bullet}).\] 
Let $\alpha\in\text{Im}(\bigoplus\limits_{F\in\mcF} r_{\mrF})\subseteq\bigoplus\limits_{F\in\mcF}H^*_T(M_\mrF)$ and $\alpha(1)$ be the element in $\text{Im}(\bigoplus\limits_{F\in\mcF(1)} r_{\mrF})$ by removing the $F_\bullet$-component of $\alpha$, where $\bigoplus\limits_{F\in\mcF(1)} r_{\mrF}: H^*_T(M(1))\to\Big(\bigoplus\limits_{F\in\mcF(1)}H^*_T(M(1)_{\mrF})\Big).$ Therefore we can express $\alpha$ as $\alpha=\Big(\alpha(1),\alpha_{F_\bullet}\Big).$
Then by induction hypothesis we have 
\[\alpha(1)\in\mathbf{k}[Q(1)].\]
Set \[\alpha(1)=\sum\limits_{F\in\mcF(1)}\theta(F),\] where $\theta(F)$ is a $F$-face element in $\Big(\bigoplus\limits_{F\in\mcF(1)}H^*_T(M(1)_{\mrF})\Big).$ Note that we can extend $\theta(F)$ to an $F$-face element $\widetilde{\theta(F)}$ in $\mathbf{k}[Q]$ uniquely as follows. 
\[
\widetilde{\theta(F)}=\bigg(\theta(F),\phi_{FF_\bullet}\Big(\theta(F)_F\Big)\bigg)
\]
where $\theta(F)_F$ is the $F$-component of $\theta(F)$ and $\phi_{FF_\bullet}$ is given in (\ref{def of phiFE for new definition}). Therefore we have
\begin{equation}\label{alpha(1),alphabullet}
\begin{split}
\Big(\alpha(1),\alpha_{F_\bullet}\Big)-\sum\limits_{\mrF\in\mathring{\mcF}(1)}\widetilde{\theta(F)}&=(\alpha(1),\alpha_{F_\bullet})-\bigg(\sum\limits_{F\in\mcF(1)}\theta(F),\sum\limits_{F\in\mcF(1)}\phi_{FF_\bullet}\Big(\theta(F)_F\Big)\bigg)\\
&=\bigg(0,\alpha_{F_\bullet}-\sum\limits_{F\in\mcF(1)}\phi_{FF_\bullet}\Big(\theta(F)_F\Big)\bigg).\\
\end{split}
\end{equation}
Lemma \ref{one key lemma} implies that 
$
\sum\limits_{\mrF\in\mathring{\mcF}(1)}\widetilde{\theta(F)}\in\text{Im}(\bigoplus\limits_{F\in\mcF}r_\mrF),
$ so we have 
\[
\bigg(0,\alpha_{F_\bullet}-\sum\limits_{F\in\mcF(1)}\phi_{FF_\bullet}\Big(\theta(F)_F\Big)\bigg)\in\text{Im}(\bigoplus\limits_{F\in\mcF}r_\mrF).\]
Considering the following commutative diagram which combines commutative diagrams (\ref{commutative diagram for HT*(M) and HT*(M(1)) and HT*(XF)}) and (\ref{commutative diagram for r(1)xF-kappa F})
\begin{center}
\begin{tikzcd}
 H^*_T(M)   \arrow[d, "{\bigoplus\limits_{F\in\mcF} r_{\mrF}}"]\arrow[r, "{}"]  &  H^*_T(M(1))\oplus H^*_T(X_{\mrF_\bullet})  \arrow[d, "{(\bigoplus\limits_{F\in\mcF(1)} r_{\mrF})\bigoplus id}"]\arrow[rr, "{r_1-\kappa_{\mrF_\bullet}}"]  && H^*_T(X_{\mrF_\bullet}\setminus M_{\mrF_\bullet})\arrow[d, "{r}"]\\
\bigoplus\limits_{F\in\mcF}H^{*}_T(M_{\mrF})    \arrow[r, "{}"]& \Big(\bigoplus\limits_{F\in\mcF(1)}H^*_T(M(1)_{\mrF})\Big)\bigoplus H^*_T(X_{\mrF_\bullet}) \arrow[rr, "{\bigoplus\limits_{F\in\mcF(1)}(h^1_\mrF-h^2_\mrF)}"]&& \bigoplus\limits_{F\in\mcF(1)}H^*_T(M(1)_{\mrF}\cap X_{\mrF_{\bullet}}),\end{tikzcd}
\end{center}
we have 
\[\left(0,\pi_{\mrF_\bullet}^*\bigg(\alpha_{F_\bullet}-\sum\limits_{F\in\mcF(1)}\phi_{FF_\bullet}\Big(\theta(F)_F\Big)\bigg)\right)\in\ker\bigoplus\limits_{F\in\mcF(1)}(h^1_\mrF-h^2_\mrF),\]
since the upper row in this commutative diagram is exact. Therefore we have 
\[
(\bigoplus_{F\in\mcF(1)}h^2_{\mrF})\left(\pi_{\mrF_\bullet}^*\bigg(\alpha_{F_\bullet}-\sum\limits_{F\in\mcF(1)}\phi_{FF_\bullet}\Big(\theta(F)_F\Big)\bigg)\right)=0,
\]
which means 
\[
\left(\pi_{\mrF_\bullet}^*\bigg(\alpha_{F_\bullet}-\sum\limits_{F\in\mcF(1)}\phi_{FF_\bullet}\Big(\theta(F)_F\Big)\bigg)\right)\Bigg|_{X_{\mrF_\bullet}\setminus M_{\mrF_\bullet}}=0.
\]
By Lemma \ref{kernel of kappa F generated by Thom class}, we have 
\[
\left(\pi_{\mrF_\bullet}^*\bigg(\alpha_{F_\bullet}-\sum\limits_{F\in\mcF(1)}\phi_{FF_\bullet}\Big(\theta(F)_F\Big)\bigg)\right)\in\Big(\tau_{F_\bullet}(X_{\mrF_\bullet})\Big),
\]
where $\Big(\tau_{F_\bullet}(X_{\mrF_\bullet})\Big)$ is the ideal, generated by $\tau_{F_\bullet}(X_{\mrF_\bullet})$, of $H^*_T(X_{\mrF_\bullet}).$ Hence we have 
\[\alpha_{F_\bullet}-\sum\limits_{F\in\mcF(1)}\phi_{FF_\bullet}\Big(\theta(F)_F\Big)=\Big(\tau_{F_\bullet}(M_{\mrF_\bullet})\Big)\]
which shows that $\bigg(0,\alpha_{F_\bullet}-\sum\limits_{F\in\mcF(1)}\phi_{FF_\bullet}\Big(\theta(F)_F\Big)\bigg)$ is an $F_\bullet$-face element. This implies that  $\alpha\in\mathbf{k}[M/T]$ by (\ref{alpha(1),alphabullet}), so $\text{Im}(\bigoplus\limits_{F\in\mcF} r_{\mrF})\subseteq \mathbf{k}[M/T].$
\end{proof}

\vspace{2mm}
\red{\begin{theo}\label{main theorem}
Let $M$ be a $2n$-dimensional open torus manifold with locally standard action. Then 
\[\bigoplus\limits_{F\in\mcF}r_\mrF: H^*_T(M)\to \mathbf{k}[M/T]\]
is an isomorphism, where $\mcF$ is the set of faces of $M/T.$
\end{theo}}
\begin{proof}
By Proposition \ref{injective observation}, we have 
\[H^*_T(M)\cong\text{Im}\left(\bigoplus\limits_{F\in\mcF} r_{\mrF}\right).\]
By Lemma \ref{one key lemma} and Lemma \ref{another key lemma}, we have 
\[
\text{Im}\left(\bigoplus\limits_{F\in\mcF} r_{\mrF}\right)=\mathbf{k}[M/T],
\]
so we obtain the isomorphism
\[H_T^*(M)\cong\mathbf{k}[M/T].\]
\end{proof}

\vspace{2mm}
\section{Towards to $H^*(BT)$-algebra structure of $H^*_T(M)$}\label{H(BT) algebra structure}
In this section we study the $H^*(BT)$-algebra structure of $H^*_T(M).$ Since $H^*(BT)$ is generated by $H^2(BT)$, it suffices to study $\pi^*(u)$ for $u\in H^2(BT)$, where $\pi^*$ is induced by  the projection map $\pi: ET\times_TM\to BT$ in the Borel construction. It suffices to study $
(\bigoplus\limits_{F\in\mcF} r_{\mrF})\Big(\pi^*(u)\Big)
$ because of Proposition \ref{injective observation}. Applying Theorem \ref{main theorem}, we have\[
(\bigoplus\limits_{F\in\mcF} r_{\mrF})\Big(\pi^*(u)\Big)=\sum\limits_{F\in\mcF}\theta(F)
\]
where $\theta(F)$'s are $F$-face elements in $\mathbf{k}[M/T].$ Note that $\deg(u)=2$ and $\deg(\theta(F))\ge 4$ if codimension of $F$ is larger than $1$, so we have
\[
(\bigoplus\limits_{F\in\mcF} r_{\mrF})\Big(\pi^*(u)\Big)=\theta(Q)+\sum\limits_{i=1}^m\theta(F_i)
\]
where $F_1,\dots, F_m$ are all the facets of $Q=M/T$. By the definition of face elements we have $\theta(Q)_Q\in H_T^2(M_{\mathring{Q}})\cong H^2(\mathring{Q})\cong H^2(Q)$ and $\theta(F_i)_{F_i}=a_i\tau_i(M_{\mrF_i})$, where $a_i\in\bf{k}$ and
 $\tau_i(M_{\mrF_i})=\tau_i|_{M_{\mrF_i}}$. Note that $\theta(Q)_Q$ and $a_i$'s are linearly dependent on $u$ where $i\in\{1,\cdots, m\}$, so we summarize above arguments as the following proposition. 
\begin{proposition}
Let $\{F_1,\dots, F_m\}$ be the set of facets of $M/T$.  To each $i\in\{1,\dots, m\}$, there is a unique element $v_i\in H_2(BT;\mathbf{k})$ 
such that for any $u\in H^2(BT;\mathbf{k})$ we have
\[(\bigoplus\limits_{F\in\mcF} r_{\mrF})\Big(\pi^*(u)\Big)=\theta(Q)+\sum\limits_{i=1}^m\langle v_i, u\rangle\tau_i(M_{\mrF_i}),\]
where $\theta(Q)_Q\in\text{Hom}_{\mathbf{k}}(H^2(BT), H^2(\mathring{Q}))$ and 
$\langle,\rangle$
 is the natural pairing between cohomology and homology.
\end{proposition}

\begin{rema}
$\theta(Q)_Q$ in the above proposition reflects the principal $T$-bundle $M_{\mathring{Q}}=\mathring{M}\to \mathring{M}/T$ because of the following commutative diagram.
\end{rema}
\begin{equation*}
\begin{CD}
ET\times M @<{}<<ET\times\mathring{M}@>>>\mathring{M}\\
@V {}VV @VV{}V@VVV\\
ET\times_T M@<{}<<ET\times_T\mathring{M}@>>>\mathring{M}/T\\
@V {}VV @VV{}V\\
BT@=BT
\end{CD}
\end{equation*}
where $ET\times_T\mathring{M}\to\mathring{M}/T$ is a homotopy equivalence.

\vspace{3mm}
\section{On equivariant characteristic classes of $M$}\label{Sect: On equivariant characteristic classes of M}
In this section we express the equivariant Stiefel-Whitney classes and Pontrjagin classes of $M$ in terms of $H^*(M/T)$ and $\tau_i$'s.

\vspace{2mm}
\begin{lemma}\label{equvariant sw class=pull back the sw class of orbit space}
Let $\overline{q}: ET\times_T M_\mrF\to\mrF$ be the second projection in the Borel construction\footnote{In this section the definition of $\overline{q}$ is slightly different from the previous sections}. Then 
\[w^T(M_{\mrF})=\overline{q}^*w(\mrF)\]
where $w^T(M_{\mrF})$ and $w(\mrF)$ denote equivariant total Stiefel-Whitney class of $M_{\mrF}$ and total Stiefel-Whitney class of $\mrF$ respectively.
\end{lemma}
\begin{proof}
Since $q: M_\mrF\to \mrF$ is a principal $T/T_F$ fiber bundle, we have 
\[
TM_\mrF\cong q^*T\mrF\oplus\nu_f,
\]
where $\nu_f$ is the subvector bundle of $TM_F$ along the fiber direction. The above bundle isomorphism is $T$-equivariant. Indeed $T$ acts on the fiber part of $q^*T\mrF$ trivially and acts on $\nu_f$ though the projection $T\to T/T_F$. Note that $\nu_f$ is $T$-equivariant isomorphism to a trivial bundle. Indeed we can assign $T$-invariant vector fields along the fiber direction of $M_\mrF$ as follows. Let $v_1,\cdots, v_l$ be a basis of $\text{Lie}(T/T_F)$, where $\text{Lie}(T/T_F)$ means the Lie algebra of $T/T_F$. For each $p\in M_\mrF$, we define 
\[\widetilde{v_i}(p):=\frac{d(\text{exp}(tv_i)\cdot p)}{dt}\Big|_{t=0}.\]
Then $\widetilde{v_i}$'s are $T$-invariant vector fields along the fiber direction. This implies that 
\[
w^T(\nu_f)=1.
\] 
Since $T$ acts on fiber direction of $q^*T\mrF$ trivially, 
we have 
\[
w^T(q^*T\mrF)=\overline{q}^*w(\mrF).\]
Therefore 
\[w^T(M_\mrF)=w^T(q^*T\mrF)\cdot w^T(\nu_f)=\overline{q}^*w(\mrF).\]
\end{proof}

\vspace{2mm}
\begin{corollary}\label{local SW classes}
Let $F$ be a codimension $k$ face of $M/T$ and a connected component of $F_{i_1}\cap\cdots\cap F_{i_k}$, where $F_{i_j}$'s are facets of $M/T$. Then  
\[
w^T(X_{\mrF})=(\pi_{\mrF}^*\circ\overline{q}^*)\Big(w(\mrF)\Big)\cdot (1+\tau_{i_1}|_{X_\mrF})\cdots(1+\tau_{i_k}|_{X_\mrF}),
\]
where $\pi_{\mrF}^*$ is induced by the map $\pi_\mrF: X_{\mrF}\to M_{\mrF}$ and $\overline{q}^*$ is induced by $\overline{q}: ET\times_T M_\mrF\to\mrF$ the second projection in the Borel construction.
\end{corollary}
\begin{proof}
Since we can identify the $T$-equivariant normal bundle $\nu_F$ with a $T$-invariant tubular neighbourhood $X_\mrF$ of $M_\mrF$ in $M$, we have the $T$-equivariant decomposition $TX_\mrF\cong\pi^*_{\mrF}(TM_\mrF)\oplus \pi^*_{\mrF}(\nu_\mrF)$. Hence we have 
\[
w^T(X_\mrF)=\pi_\mrF^*\Big(w^T(M_\mrF)\Big)\cdot\pi_\mrF^*\Big(w^T(\nu_\mrF)\Big).
\]
Since for each $p\in M_F$, $M_{F_{i_1}},\dots, M_{F_{i_k}}$ intersects at $p$ transversely, we have 
\red{$
\nu_F\cong\oplus_{j=1}^k(\nu_{F_{i_j}}|_{M_F})$
as $T$-invariant vector bundles. Therefore 
$
\nu_\mrF\cong\oplus_{j=1}^k(\nu_{F_{i_j}}|_{M_\mrF}),$ }which implies that 
\[
w^T(\nu_\mrF)=(1+\tau_{i_1}|_{M_\mrF})\cdots(1+\tau_{i_k}|_{M_\mrF}).
\]
Since $\pi_\mrF: X_\mrF\to M_\mrF$ is a $T$-equivariant homotopy equivalence, we have 
\[\pi_\mrF^*\Big(w^T(\nu_\mrF)\Big)=(1+\tau_{i_1}|_{X_\mrF})\cdots(1+\tau_{i_k}|_{X_\mrF}).\]
By Lemma \ref{equvariant sw class=pull back the sw class of orbit space}, we have 
\[
w^T(M_\mrF)=\overline{q}^*w(\mrF).
\]
Hence the equality\[
w^T(X_{\mrF})=(\pi_{\mrF}^*\circ\overline{q}^*)\Big(w(\mrF)\Big)\cdot (1+\tau_{i_1}|_{X_\mrF})\cdots(1+\tau_{i_k}|_{X_\mrF})\]
follows from the above calculations.
\end{proof}

\vspace{2mm}
Let $\mathring{M}$ be the free part of the action on $M$. Namely $\mathring{M}=q^{-1}(\mathring{Q}),$ where $\mathring{Q}$ is the relative interior part of $Q=M/T$ and $\mathring{Q}$ is homotopy equivalent to $Q$.
Then we can express the equivariant Stiefel-Whitney classes of $M$ in terms of $w(\mathring{M}/T)$ and $\tau_i$'s, where $\tau_i$ is given at the end of Section \ref{basic on torus manifolds with locally standard action}.
\begin{theo}\label{theorem for sw classes}
Let $M$ be an open torus manifold with locally standard action and $\{F_1,\cdots, F_m\}$ be the set of facets of $M/T$. Then 
\[
w^T(M)=h_M^*\Big(w(\mathring{M}/T)\Big)\cdot(1+\tau_1)\cdot(1+\tau_2)\cdots(1+\tau_m) ,
\]
where $h_M^*=\overline{q}^* \circ(\iota^*)^{-1}$, $\overline{q}^*$ is induced by $\overline{q}: ET\times_TM\to M/T$ the second projection in the Borel construction and $\iota^*$ is induced by the inclusion map
$\iota: \mathring{M}/T\to M/T$.
\end{theo}
\begin{proof}
We prove this theorem by induction on $^\sharp\mathcal{F}.$ 

When $^\sharp\mathcal{F}=1$, $M=\mathring{M}$. Since $\mathring{M}$ is a principal $T$-bundle over $\mathring{M}/T$, we have 
\[
w^T(M)=w^T(\mathring{M})=\overline{q}^*w(\mathring{Q})=h_M^*\Big(w(\mathring{M}/T)\Big)
.\]

Assume that $^\sharp\mcF\ge 2$. Let $F_{\bullet}$ be a minimal face of $Q=M/T$ and $M(1):=M\setminus M_{\mrF_{\bullet}}$. Then by Lemma \ref{sharp F(1)=sharp F-1}, we have $^\sharp\mcF(1)\le^\sharp\mcF-1$, where $
\mcF(1)~\text{is the set of faces of $M(1)/T.$}$ Set \[\eta=h_M^*\Big(w(\mathring{M}/T)\Big)\cdot(1+\tau_1)\cdot(1+\tau_2)\cdots(1+\tau_m).\] To show that  $\eta=w^T(M)$, it is enough to prove that 
\begin{equation}\label{restriction of eta on M(1) and XFbullet}
\eta|_{M(1)}=w^T(M(1))\quad\text{and}\quad\eta|_{X_{F_\bullet}}=w^T(X_{F_\bullet})
\end{equation} 
since 
\[H^*_T(M)\to H_T^*(M(1))\bigoplus H^*_T(X_{F_\bullet})\]
is injective by short exact sequence (\ref{split into short exact sequence}). Since the number of faces of $M(1)/T$ is less then $M/T$, by induction hypothesis we have 
\[
w^T(M(1))=h_{M(1)}^*\bigg(w\Big(\mathring{M(1)}/T\Big)\bigg)\cdot\Big(1+\tau_1|_{M(1)}\Big)\cdot\Big(1+\tau_2|_{M(1)}\Big)\cdots\Big(1+\tau_m|_{M(1)}\Big) 
\]
where $h_{M(1)}^*$ is defined similarly to $h_M^*$. On the other hand,
\[
\eta|_{M(1)}=h_M^*\Big(w(\mathring{M}/T)\Big)\Big|_{M(1)_T}\cdot(1+\tau_1|_{M(1)})\cdots(1+\tau_m|_{M(1)}),
\]
where $M(1)_T$ is the Borel construction of $M(1)$. By the following commutative diagram
\begin{equation*}
	\begin{CD}
	ET\times_T\mathring{M}@>{}>>	ET\times_T M(1) @>{}>> ET\times_T M\\
	@VV{}V	@VV{}V @VV{}V\\
	\mathring{M}/T@>{}>>M(1)/T@>{}>> M/T,
	\end{CD}
\end{equation*}
we have 
\[
h_M^*\Big(w(\mathring{M}/T)\Big)\Big|_{M(1)_T}=h_{M(1)}^*\Big(w(\mathring{M}/T)\Big)=h_{M(1)}^*\bigg(w\Big(\mathring{M(1)}/T\Big)\bigg)\]
where the second equality follows from $\mathring{M(1)}=\mathring{M}$. Therefore we have 
\[
\eta|_{M(1)}=w^T(M(1)).
\]

Next we prove $\eta|_{X_{F_\bullet}}=w^T(X_{F_\bullet}).$ Set $F_\bullet$ a connected component of $F_{i_1}\cap\cdots\cap F_{i_k}$, where $F_{i_j}$'s are facets of $M/T$. Then 
\[
\begin{split}
\eta|_{X_{F_\bullet}}&=h_M^*\Big(w(\mathring{M}/T)\Big)\Big|_{X_{F_\bullet}}\cdot(1+\tau_1|_{X_{F_\bullet}})\cdots(1+\tau_m|_{X_{F_\bullet}})\\
&=h_M^*\Big(w(\mathring{M}/T)\Big)\Big|_{X_{F_\bullet}}\cdot(1+\tau_{i_1}|_{X_{F_\bullet}})\cdots(1+\tau_{i_k}|_{X_{F_\bullet}}).
\end{split}
\]
By the following commutative diagram  

\begin{center}
	\begin{tikzcd}
		ET\times_{T}\mathring{X}_{F_\bullet} \arrow[ddd] \arrow[dr] \arrow[rrr]  &  & & ET\times_{T}{X_{F_\bullet}} \arrow[dl] \arrow[ddd]\\
		& \mathring{X}_{F_\bullet}/T \arrow[d] \arrow[r] & {X_{F_\bullet}}/T  \arrow[d] & \\
		& \mathring{M}/T \arrow[r] & M/T  & \\
		ET\times_{T}\mathring{M}  \arrow[ur]\arrow[rrr]& & & ET\times_{T}M \arrow[ul],
	\end{tikzcd}
\end{center}

\noindent we have 
\[
h_M^*\Big(w(\mathring{M}/T)\Big)\Big|_{X_{F_\bullet}}=h_{X_{F_\bullet}}^*\Big(w(\mathring{M}/T)\Big|_{\mathring{X}_{F_\bullet}/T}\Big)=h_{X_{F_\bullet}}^*\Big(w(\mathring{X}_{F_\bullet}/T)\Big),
\]
where $\mathring{X}_{F_\bullet}$ is the free part of $T$-action on $X_{F_\bullet}$.
Therefore we obtain  
\[
\eta|_{X_{F_\bullet}}=w^T(X_{F_\bullet}).
\]
The above argument show that (\ref{restriction of eta on M(1) and XFbullet}) holds, which means $\eta=w^T(M)$.

\end{proof}

Let $p^T(M)$ and  be the equivariant total Pontrjagin \red{class} of $M$ and $p(\mathring{M}/T)$ be the total Pontrjagin \red{class} of $\mathring{M}/T$. With the same arguments we can obtain the formula for equivariant Pontrjagin classes of $M$ as follows. 
\begin{theo}Let $M$ be an open torus manifold with locally standard action and $\{F_1,\cdots, F_m\}$ be the set of facets of $M/T$. Then 
\[
p^T(M)=h_M^*\Big(p(\mathring{M}/T)\Big)\cdot(1+\tau_1^2)\cdot(1+\tau_2^2)\cdots(1+\tau_m^2) ,
\]
where $h_M^*=\overline{q}^* \circ(\iota^*)^{-1}$, $\overline{q}^*$ is induced by $\overline{q}: ET\times_TM\to M/T$ the second projection in the Borel construction and $\iota^*$ is induced by the inclusion map $\iota: \mathring{M}/T\to M/T$.
\end{theo}


\begin{thebibliography}{99}
\bibitem{ay-ma-pa-ze-17}
A. Ayzenberg, M. Masuda, S. Park and H. Zeng, {\em{Cohomology of toric origami manifolds
with acyclic proper faces}}, J. Symplectic Geom. {\bf{15}} (2017), no. 3, 645--685.


\bibitem{ba-be-co-gi-10}
A. Bahri, M. Bendersky, F. Cohen and S. Gitler, {\em{The Polyhedral Product Functor: a
method of computation for moment-angle complexes, arrangements and related spaces}}, Adv. Math. {\bf{225}} (2010) no. 3, 1634--1668. 

\bibitem{bre-72}
G.E. Bredon, {\em{Introduction to compact transformation groups}}, Academic press, New
York–London, 1972.

\bibitem{bu-pa-15}
V. M. Buchstaber and T. E. Panov, {\em{Toric Topology}}, Mathematical Surveys and Monographs, {\bf{204}}, American Mathematical Society, Providence, RI, 2015

\bibitem{ca-gu-pi-11}
A. Cannas da Silva, V. Guillemin and A. R. Pires, {\em{Symplectic Origami}}, Int. Math. Res. Not. 2011 (2011), 4252--4293. 
\bibitem{co-li-sc-11}
D. A. Cox, J. B. Little and H. K. Schenck, {\em{Toric varieties}}, Graduate Studies in
Mathematics, {\bf{124}}, American Mathematical Society, Providence, RI, 2011.

\bibitem{da-ja-91}
M. W. Davis and T. Januszkiewicz, {\em{Convex polytopes, Coxeter orbifolds and
torus actions}}, Duke Math. J. {\bf{62}} (1991), no. 2, 417--451.


\bibitem{fu-93}
W. Fulton, {\em{An Introduction to Toric varieties}}, Ann. of Math. Studies, {\bf{113}}, Princeton Univ. Press, Princeton, N.J., 1993. 

\red{\bibitem{gu-mi-pi-sc-15}
V. Guillemin, E. Miranda, A.R. Pires and G. Scott, {\em{Toric actions on b-symplectic manifolds}}, Int. Math.
Res. Not. {\bf{14}} (2015), 5818--5848.}



\bibitem{ha-ma-03}
A. Hattori and M. Masuda, {\em{Theory of multi-fans}}, Osaka J. Math. {\bf{40}} (2003),1--68.

\bibitem{ho-pi-13}
T. Holm and A. R. Pires, {\em{The topology of toric origami manifolds}}, Math. Research Letters {\bf{20}} (2013), no. 2, 885--906.

\bibitem{is-fu-ma-13}
H. Ishida, Y. Fukukawa and M. Masuda, {\em{Topological toric manifolds}}, Moscow
Math. J. {\bf{13}} (2013), no. 1, 57--98.

\bibitem{ma-99}
M. Masuda, {\em{Unitary toric manifolds, multi-fans and equivariant index}}, Tohoku Math. J {\bf{51}} (1999), no. 2, 237--265.

\bibitem{ma-pa-06}
M. Masuda and T. Panov {\em{On the cohomology of torus manifolds}}, Osaka J.  Math. {\bf{43}} (2006), 711--746.

\bibitem{po-sa-15}
M. Poddar and S. Sarkar, {\em{A class of torus manifolds with nonconvex orbit space}}, Proc. Amer.
Math. Soc. {\bf{143}} (2015), 1797--1811.

\bibitem{yu-21}
L. Yu, {\em{A generalization of moment-angle manifolds with non-contractible orbit space}}, arXiv:2011.10366v6.

\bibitem{yo-11}
T. Yoshida, {\em{Local torus actions modeled on the standard representation}}, Adv. 
Math. {\bf{227}} (2011), 1914--1955.
\end{thebibliography}
\end{document}